\documentclass[11pt]{amsart} \usepackage{amsmath, amsthm,url,
  amssymb,setspace,epsfig, hyperref, mathrsfs,fontenc, notation,
  mathscinet, fullpage}

\usepackage[all]{xy}

\begin{document}

\title{Geometrization of principal series representations of reductive
  groups}

\author{Masoud Kamgarpour} 
\address{Max-Planck-Institut f\"ur
  Mathematik, Vivatsgasse 7, 53111 Bonn, Germany}
\email{masoudkomi@gmail.com}

\author{Travis Schedler}
\address{MIT Department of Mathematics, Room 2-172,
77 Massachusetts Avenue
Cambridge, MA 02139-4307}
\email{trasched@gmail.com}

\thanks{M.K. acknowledges the hospitality of the University of
  Calgary. T.S. was supported by an AIM Fellowship and NSF grant
  DMS-0900233.}



\begin{abstract}
  In geometric representation theory, one often wishes to describe
  representations realized on spaces of invariant functions as trace
  functions of equivariant perverse sheaves.  In the case of principal
  series representations of a connected split reductive group $G$ over
  a local field, there is a description of families of these
  representations realized on spaces of functions on $G$ invariant
  under the translation action of the Iwahori subgroup, or a suitable
  smaller compact open subgroup, studied by Howe, Bushnell and Kutzko,
  Roche, and others.  In this paper, we construct categories of
  perverse sheaves whose traces recover the families associated to
  regular characters of $T(\Fq\dblsqbrs{t})$, and prove conjectures of
  Drinfeld on their structure. We also propose conjectures on the
  geometrization of families associated to more general characters.
  \end{abstract}
\maketitle

{\small \setcounter{tocdepth}{1} \tableofcontents}

\section{Introduction: main results and
  conjectures} \label{s:introduction}

To every complex of constructible sheaves on a variety over a finite
field, Grothendieck attached the trace of Frobenius function on its
rational points. He then initiated a program to study geometric (or
sheaf-theoretic) analogues of various classical constructions on
rational points.  In this article, we study geometric analogues of
principal series representations of $G(\Fq\dblparens{t})$, where $G$
is a connected split reductive group over $\Fq\dblsqbrs{t}$. Our main
theorems, stated in \S \ref{sss:theorems}, concern geometrizing the
principal series representations that are associated to characters of
the torus $T(\Fq\dblparens{t})$ whose restrictions to
$T(\Fq\dblsqbrs{t})$ are \emph{regular}, i.e., have trivial
stabilizers under the Weyl group action. In \S \ref{ss:conjectures} we
present conjectures concerning the geometrization of more general
families of principal series representations.
 
The geometric objects we study are certain (twisted equivariant)
perverse sheaves on quotients of the loop group of $G$. Before getting
into details, let us mention two notable features of this work. First,
the quotients we consider are not, in general, proper. As far as we
know, considering perverse sheaves on non-proper quotients of the loop
group is not standard in geometric representation theory. Second, some
of the perverse sheaves that arise in the regular setting turn out to
be \emph{clean} (see Theorem \ref{t:main}). This means that these
perverse sheaves are the extensions by zero of shifted local systems
on certain locally closed subvarieties. This fact is a reflection of
the particularly simple Hecke algebras that arise in the regular case.

\subsection{Motivation} \label{ss:motivation} Let $\Fq$ be a finite
field of order $q$, $F=\Fq\dblparens{t}$, $\cO=\Fq\dblsqbrs{t}$, and
$\fp = t\cO \subseteq \cO$.  Let $G$ be a connected split reductive
group over $\cO$ and $T$ be a maximal split torus.  Choose a Borel
subgroup $B$ containing $T$. A principal series representation of
$G(F)$ is a representation obtained by parabolic induction of a
character of $T(F)$. Note that a principal series representation is,
roughly speaking, realized on a space of twisted functions on
$G(F)/B(F)$. It is well known that $G(F)$ and $B(F)$ are the sets of
$\Fq$-points of group ind-schemes $\bG$ and $\bB$ over
$\Fq$. Therefore, the naive geometric analogue of principal series
representations should be perverse sheaves on $\bG/\bB$. The problem
is that the latter ind-scheme is not an inductive limit of schemes of
finite type (we will henceforth call this property \emph{ind-finite
  type}).  For this reason, the category of perverse sheaves on
$\bG/\bB$ is not well understood.\footnote{For some results regarding
  perverse sheaves on $\bG/\bB$ see, e.g., \cite{Finkelberg98} and
  \cite{Finkelberg99}.} This issue is the source of much difficulty in
geometric representation theory.

\subsubsection{A family of unramified representations} 
One way to overcome this difficulty is to geometrize representations
in families. As an example, let us consider the geometrization of
unramified (principal series) representations.  Let
$\sW^c=\ind_{G(\cO)}^{G(F)} \bone$, where $\bone$ denotes the trivial
character.  We think of $\sW^c$ as a family of unramified (principal
series) representations. The endomorphism ring of this family
identifies with the spherical Hecke algebra
$\sH^c=\sH(G(F),G(\cO))$. The Satake isomorphism states that
$\sH^c\iso \K_0(\Rep(\hG))$, where the latter is the Grothendieck
group of the category of finite dimensional rational representations
of the dual (complex reductive) group of $G$.

\subsubsection{Geometrizing the unramified family} It is known that
$G(\cO)$ (resp. $G(F)$) is the group of $\Fq$-points of a proalgebraic
group $\bG_\cO$ (resp. a group ind-scheme $\bG$) over $\Fq$.  Fix a
prime $\ell$ not dividing $q$.  Let $\bGr:=\bG/\bG_\cO$ denote the
affine Grassmannian. Let
\[
\sWg^c=\Perv(\bGr), \quad \sWg^{c,\der} = \sD(\bGr), \quad \sHg^c =
\Perv_{\bG_\cO}(\bGr), \quad \sHg^{c,\der}= \sD_{\bG_\cO}(\bGr).
\]
Here $\sD$ denotes the bounded constructible derived category of
$\bQl$-sheaves and $\Perv$ denotes the subcategory of perverse sheaves
(see Appendix \S \ref{s:perverse}).  \commentout{(resp. $\sHg^c$)
  denote the category of $\ell$-adic perverse sheaves
  (resp. $\bG_\cO$-equivariant perverse sheaves) on the affine
  Grassmannian . Let $\sWg^{c,\der}$ (resp. $\sHg^{c,\der})$ denote
  the derived category $\sD(\bGr)$ (resp. $\sD_{\bG_O}(\bGr)$). }
There is a convolution functor
\[
\star: \sWg^c \times
\sHg^c \to \sWg^{c,\der}.
\]
This functor restricts to a functor $\star: \sHg^c\times \sHg^c\ra
\sHg^{c,\der}$.  Let $\LocSys(\Spec \Fq)$ denote the monoidal category
of $\ell$-adic local systems on $\Spec(\Fq)$. Note that this category
is equivalent to the category of finite dimensional continuous
$\ell$-adic representations of $\Gal(\bFq/\Fq)$. The following theorem
is known as the geometric Satake isomorphism, and is due to Lusztig
\cite{Lusscf}, Ginzburg \cite{Ginzburg00}, Mirkovi\'c and Vilonen
\cite{MV}, and Beilinson and Drinfeld \cite[\S 5]{BD}.

\begin{theorem}\label{t:satake} 
  \begin{enumerate}
  \item[(i)] The category $\sHg^c$ is closed under convolution.
  \item[(ii)] There is an equivalence of monoidal abelian categories
    $\Rep(\hG)\boxtimes \LocSys(\Spec \Fq) \iso \sHg^c$.\footnote{Most
      references work with the loop group over an algebraically closed
      field (e.g., $\bC$ or $\bFq$). In this case, $\LocSys(\Spec
      \Fq)$ disappears. On the other hand, Gaitsgory
      \cite{Gaitsgory01} works over $\Fq$; see the proof of
      Proposition 1 in \emph{op.~cit.}.}
    \end{enumerate} 
\end{theorem} 

The nontrivial part of Theorem \ref{t:satake}.(i) is that the
convolution of two objects of $\sHg^c$ is perverse (and not merely an
object of the equivariant derived category $\sHg^{c,\der}$). The next
theorem states that this remains true if one of the perverse sheaves
is allowed to lie in the larger category $\sWg^c$.

\begin{theorem}[\cite{Gaitsgory01}]\label{t:gaitsgory}
  $\sHg^c$ acts on $\sWg^c$ by convolution.
\end{theorem}

\subsubsection{The subject of this paper} 
In this paper, we consider the problem of geometrizing families of
principal series representations associated to nontrivial characters
$\bmu:T(\cO)\ra \bQlt$.  The families we consider have been studied by
Howe, Bushnell and Kutzko, Roche, and others. In particular, they
explain how to realize these families as representations induced from
characters of compact open subgroups. From the geometric point of
view, the advantage of inducing from a compact open subgroup $J$ is
that the corresponding quotient of varieties, $\bG/\bJ$, turns out to
be of ind-finite type.

Our main theorems, in the case of $G=\GL_N$, were conjectured by
Drinfeld in June 2005. Two of our main theorems are analogues of
Theorems \ref{t:satake} and \ref{t:gaitsgory} in the regular
setting. The other theorem concerns a phenomenon unique to the regular
setting: namely, that the irreducible objects of $\sHg$ turn out to be
clean.

\subsubsection{Connections to local geometric Langlands} 
Frenkel and Gaitsgory have outlined a program for geometrizing (or
categorifying) the local Langlands correspondence; see
\cite{Frenkel06} and \cite{Frenkel07}. Theorems \ref{t:satake} and
\ref{t:gaitsgory} play important roles in their description of the
unramified and tamely ramified part of this correspondence. We expect
that our main results will have applications in the wildly ramified
part of the Frenkel-Gaitsgory program.  In particular, in future work,
we hope to construct the geometric analogue of an irreducible
principal series representation as a category on which
$G\dblparens{t}$ acts.  We expect that this category is related to the
category of representations of the affine Kac-Moody algebra $\hfg$ at
the critical level via an infinite dimensional analogue of
Bernstein-Beilinson localization, as conjectured in \cite{Frenkel06}.

\subsection{Principal series representations via compact open
  subgroups: recollections}
We continue using the notation introduced at the beginning of previous
section. Henceforth, we assume that $q$ is restricted as in \S
\ref{sss:residue}. Fix a character $\bmu:T(\cO)\ra \bQlt$.

 \subsubsection{A Family of principal series representations}\label{sss:familyofRep}
 Let
\begin{equation}
  \displaystyle \Pi:=\iota_{B(F)}^{G(F)} \left(\ind_{T(\cO)}^{T(F)} \bmu \right).
\end{equation} 
Here ``$\ind$'' denotes the compact induction and $\iota$ denotes the
(unnormalized) parabolic induction. We think of $\Pi$ as the family of
principal series representations of $G(F)$ associated to characters of
$T(F)$ whose restriction to $T(\cO)$ is $\bmu$.  In the language of
\cite{Bernstein84}, $\Pi$ is a projective generator of the Bernstein
block of representations of $G(F)$ corresponding to $(T(\cO),\bmu)$.
Inducing endomorphisms, we obtain a canonical homomorphism
\begin{equation} 
\End_{T(F)}\left(\ind_{T(\cO)}^{T(F)} \bmu \right) \ra \End_{G(F)}(\Pi).
\end{equation} 
In the case that $\bmu$ is regular, one can show that this is an
isomorphism; see, e.g., \cite[\S 1.9]{Roche09}.

 \subsubsection{Realization of $\Pi$ via compact open subgroups}
 The following theorem, in its full generality, is due to
 Roche. Previous results in this direction (for $\GL_N$) were obtained
 by Howe \cite{Howe73} and Bushnell and Kutzko \cite{Bushnell98,
   Bushnell99}.
 
 \begin{theorem}[\cite{Roche98}] \label{t:Roche} There exists a
   compact open subgroup $J\subset G(F)$ containing $T(\cO)$ such that
   \begin{enumerate}
  \item[(i)] $\bmu$ extends to a character $\mu:J\ra \bQlt$. 
  \item[(ii)]There exists an isomorphism of $G(F)$-modules
    $\sW:=\ind_J^{G(F)}\mu \cong \Pi$. \label{t:howeBK}
   \end{enumerate} 
 \end{theorem}

 In the language of \cite{Bushnell98}, $(J,\mu)$ is a \emph{type} for
 the Bernstein block defined by $(T(\cO),\bmu)$.

 \begin{exam} \label{ex:Iwahori} Suppose the character $\bmu:T(\cO)\ra
   \bQlt$ factors through a character $\nu$ of $T(\Fq)$. Then one can
   (and Roche does) take $J$ to be the Iwahori subgroup $I\subset
   G(F)$. Let $T_1 < T(\cO)$ be the subgroup generated by the image of
   $1+\fp$ under all coweights, cf.~\S \ref{sss:reductive} below.
   Then, the character $\mu$ is defined to be the composition
\[
I\ra I/I_u \cong T(\Fq)\cong T(\cO)/T_1 \rar{\nu} \bQlt,
\]
where $I_u$ is the prounipotent radical of $I$. More generally,
Roche's subgroup $J$ equals the Iwahori subgroup if and only if $(\bmu
\circ \alpha^\vee): \Gm(\cO)\ra \bQlt$ factors through a character of
$\Gm(\Fq)$ for all coroots $\alpha^\vee:\bbG_m \ra T$ of $G$.
\end{exam}

\begin{exam} \label{ex:GLN} Suppose $G=\GL_N$. Identify $T(\cO)$ with
  $(\cOt)^N$ and write $\bmu=(\bmu_1,...,\bmu_N)$. Suppose the
  conductor $\cond(\bmu_i/\bmu_j)$ equals a fixed integer $n$ for all
  $1\leq i,j\leq N$. (The conductor of a character $\chi:\cOt\ra
  \bQlt$ is the smallest positive integer $c$ for which
  $\chi(1+\fp^c)=\{1\}$.) Then
  \[J = \begin{pmatrix} \cOt & \fp^{[\frac{n}{2}]} & \cdots & \fp^{[\frac{n}{2}]}  \\
    \fp^{[\frac{n+1}{2}]} & \cOt & \cdots & \fp^{[\frac{n}{2}]}  \\
    \vdots & \vdots & \ddots & \vdots \\
    \fp^{[\frac{n+1}{2}]} & \fp^{[\frac{n+1}{2}]} & \cdots & \cOt
  \end{pmatrix}.
\]
\end{exam}

\subsubsection{Endomorphism algebras} \label{sss:endalgs}
Given a group $K$, a character $\chi$ of $K$, and a space $X$ on which
$K$ acts (by a left action), we say a function $f$ is
$(K,\chi)$-invariant if
 \[
 f(k \cdot x) = \chi(k) f(x), \quad\quad \forall \, k \in K, \,\,
 \forall \, x \in X.
 \]
 In this language, $\sW=\ind_J^{G(F)}\mu$ is the space of compactly
 supported $(J,\mu^{-1})$-invariant functions on $G(F)$ with respect
 to right multiplication (the inverse is here since our convention is
 to use left actions: so this says that $f(j \cdot_R g) =
 f(g)\mu(j)^{-1}$, where $j \cdot_R g := g j^{-1}$).  Let
 $\sH:=\sH(G(F),J,\mu)$ denote the space of compactly supported
 $(J\times J,\mu\times \mu^{-1})$-invariant functions.  So $\sH$ is
 the space of functions $f:G(F)\ra \bQl$ satisfying
\[
f(jgj')=\mu(j)f(g)\mu(j'),\quad \quad \forall j\in J, \, g\in G(F).
\]
Convolution defines a right action $\star:\sW \times \sH\ra \sW$. It
is a standard fact that this action identifies $\sH$ with
$\End_{G(F)}(\sW)$. The isomorphism of $G(F)$-modules $\sW\cong \Pi$
defines an isomorphism of algebras $\sH\cong \End_{G(F)}(\Pi)$.

\subsubsection{Description of $\sH$ in the regular
  case} \label{sss:regularEndo} In view of \S \ref{sss:familyofRep},
in the case that $\bmu$ is regular, we obtain an isomorphism
\begin{equation} \label{eq:SatakeReg}
\Psi: \K_0(\Rep(\hT)) \iso  \sH
\end{equation} 
Since these algebras are commutative, one can show that $\Psi$ is
\emph{canonical}; i.e., it does not depend on the choice of 
isomorphism $\Pi\iso \sW$.  Using general results of Bushnell and
Kutzko on types, Roche \cite[\S 6]{Roche98} has proved that $\Psi$
sends each irreducible character of $\hT$ to an element of $\sH$
supported on a corresponding double coset of $J$, which determines the
image uniquely up to a nonzero constant.  In this paper, we compute
these constants explicitly; see Theorem \ref{t:clambda}.

\subsection{Geometrization in the regular case: main
  theorems} \label{ss:geometrizationRegular} In this subsection, we
geometrize Roche's family $\sW$, the Hecke algebra $\sH$, and the
action of $\sH$ on $\sW$ by convolution, \emph{in the case that the
  family $\sW$ is associated to a regular character of $T(\cO)$}.

The combinatorial description of $J$ makes it obvious that it is equal
to the group of $\Fq$-points of a connected proalgebraic group $\bJ$
over $\Fq$ (see \S \ref{sss:Jgeom} for an alternative explanation).
Using standard constructions from the sheaf-function dictionary, we
construct a one-dimensional character sheaf $\cM$ on $\bJ$ whose trace
of Frobenius function equals the character $\mu$. The ind-scheme
$\bG/\bJ$ is of ind-finite type; thus, we can consider the category of
(twisted perverse) sheaves and the bounded constructible derived
category of sheaves on this quotient. We observe that $\bG/\bJ$
\emph{is not proper} unless $J$ is the Iwahori subgroup.

We define $\sWgd$ to be the bounded $(\bJ,\cM^{-1}$)-equivariant
constructible derived category of sheaves on $\bG$. Since $\bJ$ acts
freely on $\bG$, this derived category can be defined naively; that
is, it equals the category of $(\bJ,\cM^{-1})$-equivariant complexes
of sheaves on $\bG$ with bounded constructible cohomology.  Let $\sWg$
denote the perverse heart of this triangulated category, which should
be thought of as the category of $\cM^{-1}$-twisted perverse sheaves
on $\bG/\bJ$ (we will define this more precisely in \S \ref{s:geom};
see also the remark below).

Next, define $\sHg$ as the category of $(\bJ\times \bJ,\cM\boxtimes
\cM^{-1})$-equivariant perverse sheaves on $\bG$ (again, we will
define this more precisely in \S \ref{s:geom}).  We define the
convolution with compact support as a functor $\star: \sWgd \times
\sHg \ra \sWgd$. There is also a convolution without compact support
defining a functor with the same source and target. A priori, they
need not coincide, since $\bG/\bJ$ is not proper (unless $J$ is the
Iwahori subgroup, as mentioned above). However, it turns out that the
two notions of convolution nonetheless coincide (Corollary
\ref{c:!=*}).

\begin{remark}\label{r:twperv} One can probably define
  $(\bJ,\cM^{-1})$-equivariant perverse sheaves on $\bG$ in (at least)
  three equivalent, but philosophically distinct, ways:
  \begin{enumerate}
  \item[(i)] Twist the category of perverse sheaves on
    $\bG/\bJ$ by a certain gerbe associated to the pair
    $(\bJ,\cM^{-1})$.
  
\item[(ii)] 
The local system $\cM$ becomes trivial after pulling back
  to a certain finite central cover $\tbJ\ra \bJ$. One then defines
  $(\bJ,\cM^{-1})$-equivariant perverse sheaves on $\bG$ as a certain
  full abelian subcategory of perverse sheaves on the
  ind-Deligne-Mumford stack $\bG/\tbJ$. 
    
\item[(iii)] There exists a proalgebraic normal subgroup $\bJ'<\bJ$
  such that $\cM$ is trivial on $\bJ'$ and $\bJ/\bJ'$ is a commutative
  algebraic group (Lemma \ref{l:J'algebraic}). Thus, $\cM$ is the
  pullback of a local system $\cM_0$ on the quotient algebraic group
  $\bA:=\bJ/\bJ'$. One then defines a $(\bJ,\cM^{-1})$-equivariant
  perverse sheaf on $\bG$ to be an $(\bA,\cM_0^{-1})$-equivariant
  perverse sheaf on $\bG/\bJ'$.
  
\end{enumerate} 
We only consider the last approach in the present text; see \S
\ref{s:geom}. For more details regarding the first two approaches, see
\S \ref{ss:alternativeTwist}.
\end{remark}

\subsubsection{Statements of the main results} \label{sss:theorems} We
continue to use the notation employed above. So $\sW$ is the family of
principal series representation associated to a regular
character of $T(\cO)$ and $\sH$ is the endomorphism ring of this
family. The abelian categories of perverse sheaves $\sWg$ and $\sHg$
are the geometric analogues of these spaces.  

\commentout{Recall that
  an irreducible perverse sheaf is \emph{clean} if it is isomorphic to
  the extension by zero of a shifted local system on a locally closed
  subvariety.}

\begin{theorem}\label{t:main} The simple objects of $\sHg$ are clean. 
\end{theorem} 
For a precise definition of clean, see Definition \ref{d:clean}.

\begin{theorem}\label{t:main2}
\begin{enumerate}
\item[(i)] The category $\sHg$ is closed under convolution. 
  \item[(ii)] There exists an equivalence of monoidal abelian categories 
\[
\Psi_{\geom}: \Rep \hT \boxtimes \LocSys(\Spec \Fq) \iso \sHg.
\]
\end{enumerate} 
\end{theorem}
\begin{theorem} \label{t:action}
 Convolution defines a monoidal abelian action $\star: \sWg
  \times \sHg \ra \sWg$. \label{con:Drinfeld}
\end{theorem} 

Theorems \ref{t:main2} and \ref{t:action} are the analogues, in the
regular setting, of Theorems \ref{t:satake} and \ref{t:gaitsgory},
respectively. Observe that, taking the trace of Frobenius, we obtain
the isomorphism (\ref{eq:SatakeReg}) and the action of $\sH$ on
$\sW$. For more precise statements of the above theorems and their
proofs, see \S \ref{s:conv}.
 
\begin{remark} The abelian category $\sHg$ should be the perverse
  heart of the bounded $(\bJ\times \bJ, \cM\boxtimes
  \cM^{-1})$-equivariant constructible derived category of sheaves on
  $\bG$. For a discussion of definition of this twisted equivariant
  derived category see \S \ref{sss:triangulated}. In view of cleanness
  (Theorem \ref{t:main}), it is reasonable to ask whether the correct
  bounded $(\bJ\times \bJ, \cM\boxtimes \cM^{-1})$-equivariant derived
  category in the regular setting is merely the bounded derived
  category of $\sHg$; we do not address this issue here (note that
  this statement is, however, a special case of Conjecture
  \ref{conj:affineHecke}).
\end{remark}

\begin{remark}\label{r:Roma} As mentioned in Example \ref{ex:Iwahori}, if 
  $\bmu$ factors through $T(\Fq)$, then the corresponding subgroup
  equals the Iwahori subgroup $I$. The category of perverse sheaves on
  the affine flag variety $\bG/\bI$ and the corresponding bounded
  derived category have been studied extensively; see, for instance,
  \cite{Roma09}, \cite{Roma2009}, \cite{RomaWeylI}, \cite{RomaWeylII},
  and \cite{RomaWeylIII}. Therefore, in this case, it might be
  possible to extract our results from the aforementioned (or related)
  references. We have not attempted to do this. On the other hand, we
  are not aware of any references where perverse sheaves on $\bG/\bJ$
  are studied, where $J$ is one of Roche's subgroups other than the
  Iwahori.
\end{remark}

\begin{remark}\label{r:Whittaker} 
  There is some similarity between our setup and that of
  \cite{Frenkel01}. In \emph{op.~cit.}, the authors define and study a
  category of perverse sheaves which geometrize $(U(F),
  \chi)$-invariant functions on $G(F)/G(\cO)$, where $U < B$ is the
  unipotent radical, and $\chi:U(F) \to \Fq$ is a generic character,
  which together with a fixed character $\psi: \Fq \to \bQlt$ yields a
  character $\psi \circ \chi: U(F) \to \bQlt$.
  The irreducible objects of their category are in bijection with
  dominant coweights. Their main result states that, like in our
  situation, these irreducible perverse sheaves are clean, and that
  their category is semisimple.\footnote{Although our category $\sHg$
    is not semisimple, that is an artifact of working over $\Fq$
    instead of over $\bFq$ as is done in \cite{Frenkel01}; note that,
    after base change to $\bFq$, $\sHg$ becomes semisimple in our
    setting.}  Note that, unlike our categories, their categories of
  perverse sheaves are defined in global terms, using Drinfeld's
  compactification of the moduli space of bundles on
  curves;\footnote{As explained in, e.g., \cite[\S 1.1.1]{Roma09},
    this category can probably be alternatively defined in a purely
    local way as a certain subcategory of perverse sheaves on the
    affine flag variety $\bG/\bI$.}
  also, their geometrization of the character $\chi: U(F) \to \Fq$ is
  a homomorphism $\mathbf{U} \to \mathbb{G}_a$ (rather than a
  character sheaf).
\end{remark}

\subsection{Geometrization in the non-regular case:
  conjectures} \label{ss:conjectures} The discussions of this
subsection are not used anywhere else in the paper; in particular, a
reader who is only interested in the regular case or not interested in
conjectures can skip this subsection and go to \S
\ref{ss:restrictions}.
   
Let $\bmu: T(\cO)\ra \bQlt$ be an arbitrary character.  We would like
to geometrize the family $\Pi=\Pi_{\bmu}$ of principal series
representations induced from characters of $T(F)$ whose restriction to
$T(\cO)$ is $\bmu$ (see \S \ref{sss:familyofRep} for the precise
definition of $\Pi$). By Theorem \ref{t:Roche}, $\Pi\cong
\sW=\ind_J^{G(F)} \mu$, and its endomorphism ring identifies with
$\sH=\sH(G(F),J,\mu)$.  To geometrize, we would like to replace
$(J,\mu^{-1})$- and $(J \times J, \mu \times \mu^{-1})$-invariant
functions on $G(F)$ by $(\bJ,\cM^{-1})$- and $(\bJ \times \bJ, \cM
\times \cM^{-1})$-equivariant perverse sheaves on $\bG$.  However, it
is known that, in general, the category of $(\bJ \times \bJ,\cM \times
\cM^{-1})$-equivariant perverse sheaves on $\bG$ is not closed under
convolution (this fails already in the unramified setting: the
convolution of two $\bI$-equivariant perverse sheaves on $\bG/\bI$ is
not necessarily perverse).

In what follows, we propose two remedies:
\begin{enumerate}
\item[(I)] geometrize $\sW$ and $\sH$ using the equivariant derived
  categories;
\item[(II)] provided $\bmu$ is of a special form (which we call
  parabolic), geometrize a closely related family of representations
  and its endomorphism ring using perverse sheaves.
\end{enumerate}

\subsubsection{Roche's description of $\sH$} \label{sss:affineHecke}
According to \cite[\S 6]{Roche98}, for arbitrary $\bmu$, there
exists a (possibly disconnected) split reductive group $L$ over $F$
such that
\begin{equation} \label{eq:affineHecke} \sH=\sH(G(F),J,\mu) \cong
  \sH(L(F), I_{L^\circ}),
\end{equation} 
where $L^\circ$ is the connected component of the identity of $L$ and
$I_{L^\circ}$ is an Iwahori subgroup of $L^\circ$.\footnote{Note that
  according to \S 8 of \emph{op.~ cit.}, $L^\circ$ is an endoscopic
  group of $G$ (but not necessarily a subgroup).}  Note that in the
unramified setting (i.e., $\bmu=\bone$), we have $L=L^\circ=I$; see
Example \ref{ex:Iwahori}. On the other hand, in the regular setting,
$L=L^\circ=T$; see \S \ref{sss:regularEndo}.

\subsubsection{Geometrization via the derived
  category} \label{sss:triangulated} Let $\sHg^\der$ denote the
(bounded) equivariant derived category $\sD_{(\bJ\times \bJ,
  \cM\boxtimes \cM^{-1})}(\bG)$. Since the action of $\bJ\times \bJ$
is not free, the latter is not necessarily the same as the derived
category of $(\bJ\times \bJ,\cM\boxtimes \cM^{-1})$-equivariant
complexes of sheaves with bounded constructible cohomology. We do not
know of a reference that gives a proper definition. The correct
definition can probably be obtained by modifying the approach of
Bernstein and Lunts \cite{Bernstein94} (they consider the case where
$\cM$ is trivial). Alternatively, one can try to twist triangulated
categories following \cite[\S I.2]{Reich10}.  Henceforth, we assume
that one has the correct definition of this twisted equivariant
derived category. Convolution of sheaves should then endow
$\sHg^{\der}$ with a structure of monoidal triangulated category. The
following is a conjectural geometrization of \eqref{eq:affineHecke}.

\begin{conj} \label{conj:affineHecke} There exists an equivalence of
  monoidal triangulated categories $\sHg^\der \cong
  \sD_{\bI_{\bL^\circ}} (\bL/\bI_{\bL^\circ})$.
\end{conj} 

For applications to the Langlands program, we would like to have a
description of $\sHg^\der$ in Langlands dual terms. In the case that
$L$ is connected (i.e. $L=L^\circ$), there is an answer explained to
us by Bezrukavnikov (see the very similar \cite[Theorem
4.2.(a)]{Bezicm}):
\begin{theorem}\label{t:bezrukavnikov} There exists a canonical
  equivalence of triangulated categories
  \begin{equation} \label{e:bezrukavnikov} \sD_{\bI_\bL}
    (\bL/{\bI_\bL}) \iso \sD \QCoh_{\hL} (\tcN \times^R_{\hfl} \tcN).
  \end{equation} 
  where $\tcN$ is the Springer resolution of the cone of nilpotent
  elements in $\hfl$ and the $\times^R_{\hfl}$ means one must take a
  derived (dg-algebra) fibered product.\footnote{According to
    Bezrukavnikov, the reason the fibered product is derived has to do
    with the fact that $\tcN \times_{\hfl} \tcN$ is not a complete
    intersection in $\tcN \times \tcN$. To avoid this one can work
    instead with the monodromic equivariant category
    $\sD_{\bI_{\bL,u}}(\bL/\bI_{\bL}) \iso \sD \QCoh_{\hL}
    (\tilde{\mathfrak{l}} \times_{\hfl} \tcN)$, where
    $\tilde{\mathfrak{l}} \onto \hfl$ is the Grothendieck-Springer
    resolution, and $\bI_{\bL,u}$ is the prounipotent radical of
    $\bI_\bL$. Perhaps, in our situation, one might similarly prefer
    to define $\sHgd$ as a certain monodromic equivariant category;
    if correctly defined, we could ask for an
    analogue of Conjecture \ref{conj:affineHecke} that states that the
    result is equivalent to $\sD_{\bI_{u,\bL}}(\bL/\bI_\bL)$.}
  \end{theorem} 
 
  The above conjecture (together with Bezrukavnikov's theorem) gives a
  geometrization of $\sH$, in the sense that $\K_0(\sHg^{\der}) \cong
  \sH$.  Moreover, the monoidal category $\sHg^\der$ acts on
  $\sWg^\der$ geometrizing the action of $\sH$ on $\sW$.

  \subsubsection{Parabolic characters} \label{sss:parabolic} To
  geometrize using only perverse sheaves (as opposed to the whole
  equivariant derived category) we need to consider not the family
  $\sW$, but a closely related family which we will call $\sW^c$. In
  the case $\bmu=\bone$, we defined this family in \S
  \ref{ss:motivation}. In general, we only know how to define this
  family when the character $\bmu$ is parabolic (and do not know if it
  should exist more generally).

  \begin{definition} A character $\bmu:T(\cO)\ra \bQlt$ is
    \emph{parabolic} if the stabilizer $\Stab_W(\bmu)\subseteq W$ of
    $\bmu$ is a parabolic subgroup of $W$, i.e., there exists a
    parabolic subgroup $P\subseteq G$ whose associated Weyl group is
    $\Stab_W(\bmu)$.\footnote{Equivalently, the group $L$ associated
      to $\bmu$ of \S \ref{sss:affineHecke} is the Levi of a parabolic
      subgroup of $G$.}
\end{definition}

Let $\bmu$ be a parabolic character. Let $L$ denote the Levi of the
parabolic associated to $\bmu$. Thus, $L$ is a connected split
reductive subgroup of $G$. It is easy to see that $\bmu$ extends to a
character of $L(\cO)$ which, by an abuse of notation, will also be
denoted by $\bmu$. Define a new family of principal series
representations by
\begin{equation}
\Pi^c:= \iota_{P(F)}^{G(F)}(\ind_{L(\cO)}^{L(F)} \bmu).
\end{equation} 
We think of $\Pi^c$ also as a family of principal series
representations associated to $\bmu$. We believe that $\Pi^c$ can also
be realized by inducing a character of a compact open subgroup of
$G(F)$. To this end, let $J^c:=J.L(\cO)$. One can show that
$\bmu:L(\cO)\ra \bQlt$ extends to a character $\mu^c: J^c\ra
\bQlt$. Let $\sW^c:=\ind_{J^c}^{G(F)} \mu^c$ and $\sH^c
:= \End_{G(F)}(\sW^c) = \sH(G(F),J^c,\mu^c)$.

\begin{conj} \label{c:centralFamily} 
  \begin{enumerate}
  \item[(i)] There exists an isomorphism of $G(F)$-modules $\Pi^c
    \cong \sW^c$.
\item[(ii)] There exists a ``Satake-type" isomorphism $\sH^c \cong \K_0(\Rep(\hL))$.
\end{enumerate} 
\end{conj} 
\begin{remark} It should not be difficult to prove part (ii): first,
  $\sH^c$ is a subalgebra of $\sH$, which by \eqref{eq:affineHecke} is
  isomorphic to $\sH(L(F),I_{L})$.  The latter is the affine Hecke
  algebra for $L$, whose basis consists of functions supported on
  $I_L$-double cosets labeled by the affine Weyl group of $L$.  Then,
  the $J^c$-double cosets of $G(F)$ which support functions of $\sH^c$
  should correspond to the $I_L$-double cosets of $L(F)$ labeled by
  $\Stab_W(\bmu)$-invariant coweights of $\hT$, cf.~\S
  \ref{ss:relevantCosets}.  These, in turn, are also a basis for
  $\K_0(\Rep(\hL))$, which should yield (ii).
\end{remark}
\begin{exam} We now list some cases where the above conjecture is
  known to be true:
\begin{enumerate}
\item[(i)] Every regular character $\bmu: T(\cO)\ra \bQlt$ is parabolic. In this case, the $G(F)$-modules $\sW^c$ and $\sW$ (and, therefore, the algebras $\sH^c$ and $\sH$) coincide. 
\item[(ii)] The trivial character (corresponding to the unramified
  case) is parabolic. In this case, $J=I$ and $J^c=G(\cO)$.
\item[(iii)] If $G=\GL_N$ then every character of $T(\cO)$ is
  parabolic. In this case, the above conjecture was proved in
  \cite{Howe73} (with a somewhat different choice of $J^c$).
\end{enumerate} 
\end{exam}

\subsubsection{Geometrization of families associated to parabolic
  characters} \label{sss:parabolicGeom} Let $\bmu$ be a parabolic
character and let $J^c, \mu^c, \sW^c$ and $\sH^c$ be as above. One can
show that $J^c$ is the group of points of a proalgebraic group $\bJ^c$
and that $\mu^c$ is the trace of Frobenius function of a
one-dimensional character sheaf $\cM^c$ on $\bJ^c$. Let $\sHg^c$ and
$\sWg^c$ be the abelian categories of twisted equivariant perverse
sheaves corresponding to $\sH^c$ and $\sW^c$.  Explicitly, $\sHg^c =
\Perv_{(\bJ^c \times \bJ^c, \cM^c \times (\cM^c)^{-1})}(\bG)$ and
$\sWg^c = \Perv_{(\bJ^c, (\cM^c)^{-1})}(\bG)$.
\begin{conj} \label{conj:shgc}
  \begin{enumerate}
\item[(i)] $\sHg^c$ is closed under convolution. 
\item[(ii)] There exists an equivalence of monoidal abelian categories
  $\sHg^{c} \cong \Rep \hL \boxtimes \LocSys(\Spec \Fq)$.
\end{enumerate} 
\end{conj}

\begin{conj} $\sHg^c$ acts on $\sWg^c$ by convolution.
\end{conj} 
Note that the above conjectures specialize to Theorems \ref{t:satake}
and \ref{t:gaitsgory} in the unramified case and Theorems
\ref{t:main2} and \ref{t:action} in the regular case. We think of the
equivalence of Conjecture \ref{conj:shgc} as a geometric Satake
isomorphism for $\bmu$.

\subsubsection{Central functor} 
One can ask if there is a categorification of the inclusion $\sH^c
\cong Z(\sH) \into \sH$.  In the unramified setting this was done by
Gaitsgory \cite{Gaitsgory01}:\footnote{This was, in a sense, the first
  step behind the description in Theorem \ref{t:bezrukavnikov} of
  $\sD_{\bI}(\bG/\bI)$.} there exists a functor
 \[
\sZ: \sHg^c = \Perv_{\bG_{\cO}}(\bGr) \ra \sHg = \Perv_{\bI}(\bG/\bI),
 \]
 such that $\sZ(-) \star -$ and $- \star \sZ(-)$ yield well-defined,
 isomorphic functors $\Perv_{\bG_\cO}(\bGr) \times
 \Perv_{\bI}(\bG/\bI) \to \Perv_{\bI}(\bG/\bI)$.  Considered as a
 functor to $\sD_\bI(\bG/\bI)$, $\sZ$ is monoidal. Moreover, according
 to \cite{Zhu10}, $\sZ$ can be upgraded to a monoidal functor
 $\sD_{\bG_\cO}(\bGr) \ra \sD_\bI(\bG/\bI)$, also with the property
 that $\sZ(\cF) \star -$ is isomorphic to $- \star \sZ(\cF)$, and
 which induces an isomorphism from the K-theory of the source to the
 center of the K-theory of the target.
  
\begin{conj}\label{conj:cent} 
\begin{enumerate}
\item[(i)] There exists a functor $\sZ: \sHg^{c}\ra \sHg$ such that
  $\sZ(-) \star -$ and $- \star \sZ(-)$ yield well-defined, isomorphic
  actions of $\sHg^c$ on $\sWg$.
\item[(ii)] The functor $\sZ$ can be upgraded to a monoidal functor
  $\sHg^{c,\der}\ra \sHg^\der$, such that $\sZ(-) \star -$ is
  isomorphic to $- \star \sZ(-)$, and which induces an isomorphism
  from the K-theory ring of the source to the center of the K-theory
  ring of the target.
  \end{enumerate} 
\end{conj}

Roughly speaking, the above conjectures says that we can think of
$\sHg^{c,\der}$ as a ``monoidal center'' of the triangulated category
$\sHg^\der$. The monoidal abelian category $\sHg^c$ is the perverse
heart of this monoidal center.
  
\begin{remark} Let $\bmu$ be an arbitrary character of $T(\cO)$.
  Possibly, one could still geometrize $\sH^c$ as the perverse heart
  $\sHg^c$ of some ``central'' subcategory of $\sHg^{\der}$. If so, we
  would hope that $\sHg^c$ is closed under convolution and has a
  description in Langlands dual terms. Moreover, one can ask if
  Conjectures \ref{conj:shgc} and \ref{conj:cent} still hold, at least
  when the group $L$ from \S \ref{sss:affineHecke} is connected.
   \end{remark}

   \subsection{Restrictions on the field, group, and
     coefficients} \label{ss:restrictions} In this paper, we work over
   $\Fq\dblparens{t}$, assuming that $q$ is mildly restricted
   depending on which reductive group we are considering; see \S
   \ref{sss:residue} for details. This restriction is necessary
   because it is required by \cite{Roche98} for Theorem
   \ref{t:Roche}.(ii) (we actually relax the conditions of that
   theorem in view of \cite{Yu-ctsr}: see \S \ref{sss:residue}).  We
   note that the main results of this paper and their proofs (\S$\!$\S
   \ref{s:geom} - \ref{s:conv}) also hold for $k\dblparens{t}$, where
   $k$ is an algebraically closed field of characteristic restricted
   as in \S \ref{sss:residue};
   the main caveat is that
   one does not necessarily have a $\bmu$, but only a choice of
   Roche's group $\bJ$ and an appropriate character sheaf $\cM$ and
   subgroup $\bJ' < \bJ$ on which $\cM$ is trivial; see Remark
   \ref{r:algclsd} for details. We can also work over algebraically
   closed fields of characteristic zero, at the price of making $J$ be
   necessarily the Iwahori subgroup; see Remark \ref{r:charzero}.

   In the unequal characteristic setting, one also has group
   ind-schemes (resp. group schemes) $\bG$ (resp. $\bG_\cO$) such that
   $\bG(\Fq)=G(F)$ (resp. $\bG_\cO(\Fq)=G(\cO)$). However, as far as
   we know, there is no realization of the affine Grassmannian as an
   ind-scheme of ind-finite type; see, for instance,
   \cite{Kreidl10}. Therefore, we don't know how to make sense of
   sheaves on the affine Grassmannian in the unequal characteristic.

 We consider connected split reductive groups $G$ over $\cO$. One can
 ask if there are analogues of our results for non-split groups.  In
 this regard, we note that Haines and Rostami have proved a version of
 Satake isomorphism for non-split groups \cite{Haines10}. Furthermore,
 X. Zhu has proved an analogue of Theorem \ref{t:gaitsgory} for
 non-split groups \cite{Zhu10}. In particular, for quasisplit groups,
 where principal series representations still make sense, we expect
 that our results should admit a generalization.

 We work with $\bQl$-sheaves on the \'{e}tale topology. Over $\bC$,
 one can also work with sheaves in the complex topology. In
 particular, in \cite{MV}, Mirkovi\'c and Vilonen prove the geometric
 Satake isomorphism for sheaves of $R$-modules on the complex topology
 of the loop group, where $R$ is a commutative Noetherian unital ring
 of finite global dimension. Our main results and proofs (\S$\!$\S
 \ref{s:geom} - \ref{s:conv}) also extend to this setting, at least
 when $R$ is a field; see Remark \ref{r:anycoeff} for details.

\subsection{Acknowledgements} We are indebted to V. Drinfeld for
sharing his conjectures with us and for helpful conversations.  We
thank D. Gaitsgory for helping us at crucial stages. We are grateful
to R. Howe for bringing to our attention his remarkable paper
\cite{Howe73}. A. Roche gave informative answers to many of our
questions and helped us understand the realization of representations
via compact open subgroups. Finally, we would like to thank
S. Arkhipov, D. Ben-Zvi, R. Bezrukavnikov, M. Boyarchenko,
W. Casselman, P. Etingof, J. Gordon, T. Haines, J. Kamnitzer,
C. Mautner, D. Nadler, D. Nikshych, R. Reich, P. Sally, K. Vilonen,
Z. Yun, and X. Zhu for helpful conversations.


\section{Roche's compact open subgroups} \label{s:compactGroups}

\subsection{Conventions} \label{ss:conventions} In the present
section, as well as \S \ref{s:repViaCompact}, $F$ and $\cO$ need not
be $\Fq\dblparens{t}$ and $\Fq\dblsqbrs{t}$ as we assume in the rest
of the paper; it suffices for $F$ to be a local field with ring of
integers $\cO$, unique maximal ideal $\fp$, residue field $\Fq$, and
uniformizer $t$.

\subsubsection{Reductive groups} \label{sss:reductive} Let $G$ be a
connected split reductive group over $\cO$.  Fix a split maximal torus
$T < G$. Let $\Delta\subset \Hom(T,\bbG_m)$ denote the set of roots of
$G$ with respect to $T$. Let $\Lambda = \Hom(\Gm,T)$ be the coweight
lattice.
To an element $\lambda \in \Lambda$, we associate $t^\lambda
= \lambda(t) \in T(F)$.

For every $\alpha \in \Delta$, let $u_\alpha: \bbG_a \to G$ be the
corresponding one-parameter subgroup, where $\bbG_a$ is the additive
group.  Let $U_\alpha < G$ be the image. For all $i \in \bZ$, let
$U_{\alpha,i} = u_\alpha(\fp^i) < G(F)$. Moreover, for $i \geq 1$, let
$T_i$ be the subgroup of $T(\cO)$ generated by the image of $1+\fp^i$
under all coweights, i.e., the image of $1+\fp^i$ under the natural
isomorphism of topological groups $\Lambda \otimes_{\bZ} F^\times \iso
T(F)$.  In particular, for $i \geq 1$, $T_i$ and $U_{\alpha,i}$ are
the kernels of $T(\cO) \to T(\cO/\fp^i)$ and $U_\alpha(\cO) \to
U_\alpha(\cO/\fp^i)$.

Fix a partition $\Delta = \Delta_+ \sqcup \Delta_-$ into positive and
negative roots.  Let $B=B^+$ denote the Borel subgroup defined by
$\Delta_+$ and $B^-$ denote the Borel defined by $\Delta_-$. Let
$U=U^+$ denote the unipotent radical of $B$ and let $U^-$ denote the
unipotent radical of $B^-$.  Let $\Lambda_+ \subseteq \Lambda$ denote
the subset of dominant coweights; that is,
\[
\Lambda_+:=\{ \lambda\in \Lambda \, |\, \alpha(\lambda)\geq 0 \quad
\forall \, \, \alpha\in \Delta_+\}.
\]
Then $-\Lambda_{+}$ is the set of antidominant coweights.  (Note that,
by our conventions, $\Lambda_+ \cap -\Lambda_+$ is the sublattice of
coweights whose image is in the center of $G$.)

\subsubsection{Representation theory over $\bC$ vs. $\bQl$} 
We fix, once and for all, a prime number $\ell$ not a factor of $q$,
and an isomorphism of fields $\bQl \cong \bC$. Using the isomorphism,
we carry over results regarding complex coefficients to the $\bQl$
case.

\subsection{Roche's compact open subgroup}\label{ss:rocheGroups}
Suppose that $f: \Delta \to \bZ$ is a function satisfying the
properties
\begin{enumerate}
\item[(a)] $f(\alpha) + f(\beta) \geq f(\alpha+\beta)$, whenever
  $\alpha, \beta, \alpha+\beta \in \Delta$;
\item[(b)] $f(\alpha) + f(-\alpha) \geq 1$.
\end{enumerate}

Define the following subgroups of $G(F)$: 
\begin{gather*}
  U_f := \langle U_{\alpha, f(\alpha)} \mid \alpha \in \Delta \rangle; \\
  U_{f,\alpha} := U_f \cap U_\alpha(F); \\
  J_f := \langle U_f, T(\cO) \rangle; \\
  T_f := \prod_{\alpha \in \Delta} \alpha^\vee
  (1+\fp^{f(\alpha)+f(-\alpha)}) < T(F).
\end{gather*}
Then Roche proved (based on results from \cite{BTgrcl1})

\begin{lemma}\cite[Lemma 3.2]{Roche98} \label{l:uroche} 
\begin{enumerate}
\item[(i)] $U_{f,\alpha} = U_{\alpha, f(\alpha)}$ for all $\alpha \in
  \Delta$;
\item[(ii)] The product map $\prod_{\alpha \in \Delta_{\pm}}
  U_{\alpha,f(\alpha)} \to U_f^{\pm}$ is bijective for any ordering of
  the factors in the product and any choice of sign $\pm$;
\item[(iii)] $U_f$ has the direct product decomposition $U_f = U_f^-
  T_f U_f^+$;
\item[(iv)] $J_f$ has the direct product decomposition $J_f = U_f^-
  T(\cO) U_f^+$.
\end{enumerate}
\end{lemma}

Henceforth, we assume that $f$ is fixed and write $J=J_f$ for the
corresponding compact open subgroup of $G(F)$.  The above lemma
implies that, under a suitable ordering of $\Delta$ (i.e., any
ordering of $\Delta_+$ followed by any ordering of $\Delta_-$, or vice
versa), there are direct product decompositions

 \begin{equation} \label{eq:conjLambda} t^\lambda J t^{-\lambda} = T(\cO)
   \times \prod_{\alpha \in \Delta}
   U_{\alpha,f(\alpha)+\alpha(\lambda)};
\end{equation} 

\begin{equation} \label{eq:intersectLambda} J \cap t^\lambda J
  t^{-\lambda} = T(\cO) \times \prod_{\alpha \in \Delta}
  U_{\alpha,\max\{f(\alpha),f(\alpha)+\alpha(\lambda)\}};
  \end{equation} 

  \begin{equation}\label{eq:intersect2Lambdas}
    (t^{-\lambda} U_f^+ t^\lambda) \cap (t^{-\nu} U_f^+ t^\nu) = 
    \prod_{\alpha \in
      \Delta_+} U_{\alpha,f(\alpha)-\min\{   \alpha(\lambda), 
      \alpha(\nu)           \}}.
\end{equation}

We will make frequent use of the decomposition of Lemma
\ref{l:uroche}.(iv).  For convenience, let $J^- := U_f^-, J^0 :=
T(\cO)$, and $J^+ := U_f^+$, so that $J = J^- J^0 J^+$, which is a
direct product decomposition.  We refer to this as the Iwahori
decomposition of $J$; we will also use that the decomposition remains
valid if the three factors $J^-, J^0$, and $J^+$ are rearranged in any
order.  Note that if $f$ is defined to be $0$ on the positive roots
and $1$ on the negative roots, then $J$ coincides with the Iwahori
subgroup of $G(F)$ defined by $\Delta_+$, and the above product is the
usual Iwahori decomposition.

\subsection{Relevant double cosets} \label{ss:relevantCosets} We are
particularly interested in the following special double cosets of $J$,
since, by \cite[Theorem 4.15]{Roche98}, for a regular character $\bmu:
T(\cO) \to \bQlt$, they are the only ones that support $(J \times J,
\mu \times \mu^{-1})$-invariant functions (cf.~\S \ref{sss:endalgs}),
for $J = J_{f_{\bmu}}$ and $\mu: J \to \bQlt$ as defined in
\emph{op.~cit.}, and recalled in \S \ref{sss:compactGr} below.

\begin{definition} \label{d:relevant} A {\it{relevant double coset}}
  is a double coset of $J$ in $G(F)$ of the form $Jt^\lambda J$, for
  $\lambda \in \Lambda$.
  \end{definition}  

  We now establish some elementary properties of relevant double
  cosets.
\begin{lemma} \label{l:conjJ} Suppose $\lambda\in \Lambda_{+}$.
\begin{enumerate} 
\item[(i)] $t^\lambda J^+ t^{-\lambda} \subseteq J^+$.
\item[(ii)] $t^{-\lambda} J^{-} t^\lambda \subseteq J^-$.
\item[(iii)] $t^\lambda J^0 t^{-\lambda} = J^0$. 
\end{enumerate} 
\end{lemma}

\begin{proof} (i) and (ii) follow immediately from
  \eqref{eq:conjLambda} and definition of dominant and antidominant
  coweights. (iii) is obvious.
\end{proof} 

The following proposition will be crucial for us. Particularly note
the ``miracle'' of part (c), which will be a key ingredient of the
proof of Proposition \ref{p:conv}.
\begin{proposition} \label{p:cosetMult}
\begin{enumerate}
\item[(a)] For all $\lambda\in \Lambda_{+}$, $t^{-\lambda}J
  t^{\lambda} J= t^{-\lambda} J^{+} t^{\lambda} \times J^{0}\times
  J^{-}$.
\item[(b)] For all $\lambda, \nu\in \Lambda_{+}$, $Jt^\lambda Jt^\nu
  J=Jt^{\lambda+\nu}J$.
\item[(c)] Suppose $\lambda, \nu, \kappa \in \Lambda$.  Then
  $Jt^{\kappa}J\cap Jt^{\lambda}J t^{\nu}J$ is empty unless
  $\kappa=\lambda + \nu$.
\item[(d)] For all $\lambda, \nu \in \Lambda_+$, $Jt^{\nu} Jt^{-\nu} J
  \bigcap J t^{-\lambda} J t^{\lambda} J = J$.
\end{enumerate}
\end{proposition}

\begin{proof}
  Lemma \ref{l:conjJ} implies that $t^{-\lambda} J t^\lambda
  J=(t^{-\lambda} J^{+} t^\lambda) J^{0}J^{-}$. Part (a) follows by
  observing that $t^{-\lambda} J^+ t^\lambda \subseteq U^+(F)$, and
  that $G(F) = U^+(F) T(F) U^-(F)$ is a direct product decomposition.

For (b), note that 
\[
Jt^\nu J t^\lambda J = J t^\nu (J^+J^0J^-) t^{\lambda} J = J (t^\nu
J^+ t^{-\nu}) t^\nu J^0 t^{\lambda} (t^{-\lambda} J^- t^\lambda) J
\subseteq J t^\nu J^0 t^\lambda J = J t^{\nu+\lambda} J.
  \]
The reverse inclusion is obvious. 

Next, for (c), first note that, up to the choice of positive roots
$\Delta_+ \subseteq \Delta$ (which does not change $J$ and therefore
does not change the statement), we can assume that $\lambda$ is
dominant. Then,
\begin{equation*}
  Jt^{\kappa}J \cap Jt^{\lambda}J t^{\nu}J \subseteq J t^\kappa J \cap 
  J t^\lambda J^- t^\nu J
  = J[J^+ J^0 t^\kappa J^+ J^0 \cap J^- t^\lambda J^- t^\nu J^-]J.
\end{equation*}
The last equality is easily established using the Iwahori
decomposition (and that $J^0 J^+ = J^+ J^0$ and similarly $J^0 J^- =
J^- J^0$).  Now, $J^+ J^0 t^\kappa J^+ J^0\subset t^{\kappa} J^0
U^+(F)$ and $J^- t^\lambda J^- t^\nu J^-\subset t^{\lambda+\nu}
U^-(F)$. Their intersection evidently is $\{t^{\lambda+\nu}\}$ if
$\kappa = \lambda+\nu$ and is empty otherwise.

Finally, for (d), the containment $\supseteq$ is obvious. Then,
\begin{multline*}
  Jt^{\nu} Jt^{-\nu} J \cap J t^{-\lambda} J t^{\lambda} J \subseteq J
  t^\nu J^- t^{-\nu} J \cap J t^{-\lambda} J^+ t^\lambda J \\ = J[J^-
  t^\nu J^- t^{-\nu} J^- \cap J^+ J^0 t^{-\lambda} J^+ t^\lambda J^+
  J^0]J \subseteq J[t^\nu J^- t^{-\nu} \cap t^{-\lambda} J^0 J^+
  t^\lambda]J = J. \qedhere
\end{multline*}
\end{proof}

The main application of parts (b) and (d) is
\begin{corollary}\label{c:proddom} Let $\lambda, \nu \in \Lambda_+$ or 
  $\lambda,\nu\in -\Lambda_{+}$. Then, the multiplication map defines
  a bijection $(Jt^\lambda J) \times_{J} (J t^\nu J) \ra
  Jt^{\lambda+\nu}J$.
\end{corollary}
Here, for any subsets $S_1, S_2 \subseteq G(F)$ invariant under right
and left multiplication by $J$, respectively, $S_1 \times_J S_2 := S_1
\times S_2 / ((s_1,s_2) \sim (s_1j^{-1}, js_2), \forall j \in J)$
(i.e., it is the quotient of $S_1 \times S_2$ by the inner adjoint
action of $J$).

\begin{proof}
  First of all, note that, by Proposition \ref{p:cosetMult}.(b), the
  multiplication map is surjective. To prove it is injective, suppose
  that $xy=x'y'$, with $x,x'\in Jt^\lambda J$ and $y,y'\in Jt^\nu
  J$. Then $x^{-1}x'\in Jt^{-\lambda}J t^{\lambda} J$, whereas
  $y(y')^{-1}\in Jt^\nu J t^{-\nu} J$. Since $x^{-1}x'=y(y')^{-1}$,
  Proposition \ref{p:cosetMult}.(d) implies that $x^{-1}x'\in J$.
 \end{proof}

 \subsection{Volume of relevant double cosets and
   semismallness} \label{ss:volume} In this section we prove some
 results we need about volumes of double cosets. Fix a left-invariant
 Haar measure, $\vol$, on $G(F)$ such that $J$ has measure $1$.

 \begin{lemma}\label{l:vollam} For all $\lambda \in \Lambda$, $\vol(J
   t^\lambda J) = q^{\sum_{\alpha\in \Delta_+} |\alpha(\lambda)|}$.
\end{lemma}

\begin{proof}
  By left-invariance, $\vol(J t^\lambda J)$ is the number of left
  cosets of $J$ in $J t^\lambda J$.  Since $J$ acts transitively by
  left multiplication on the set of left cosets in $J t^\lambda J$
  with stabilizer of $t^\lambda J$ equal to $J \cap (t^{\lambda} J
  t^{-\lambda})$, we obtain that $\vol(J t^\lambda J) = |J / (J \cap
  (t^{\lambda} J t^{-\lambda}))|$.  The result then
  follows from \eqref{eq:intersectLambda}. 
\end{proof}

\begin{corollary} \label{c:volumeCosetMult}
\begin{enumerate} 
\item[(i)] For all $\lambda\in \Lambda$, $\vol(J t^\lambda J) = \vol(J
  t^{-\lambda} J)$.
\item[(ii)] For all $\lambda,\nu\in \Lambda_{+}$, $\log_q
  \vol(Jt^{\lambda}Jt^{\nu}J)=\log_q \vol(Jt^{\lambda}J)+ \log_q
  \vol(Jt^{\nu}J)$.
\item[(iii)] For $\lambda,\nu\in \Lambda_+$,
$\log_{q} \vol \bigl(t^{-\lambda} J t^\lambda J \cap t^{-\nu} J t^\nu J
\bigr)= \frac{1}{2} [ \log_{q}\vol (Jt^\nu J) + \log_{q}\vol
(Jt^\lambda J) - \log_{q} \vol (Jt^{\nu-\lambda}J) ]$. 
\end{enumerate}
\end{corollary}

\begin{proof} (i) follows immediately from Lemma \ref{l:vollam}. (ii)
  follows from Lemma \ref{l:vollam} and Corollary \ref{c:proddom}. For
  (iii), note that by Proposition \ref{p:cosetMult}.(a), $t^{-\lambda}
  J t^\lambda J \cap t^{-\nu} J t^\nu J = ((t^{-\lambda} J^+
  t^\lambda) \cap (t^{-\nu} J^+ t^\nu)\bigr) J$. Thus, by
  \eqref{eq:intersect2Lambdas},
\begin{multline*}
  \log_{q} \vol \bigl( \bigl((t^{-\lambda} J^+ t^\lambda) \cap
  (t^{-\nu} J^+ t^\nu)\bigr) J \bigr) = \sum_{\alpha\in \Delta_+} \min
  \{ \alpha(\lambda), \alpha(\nu) \}
  = \sum_{\alpha\in \Delta_+} \frac{1}{2} \bigl( \alpha(\lambda) + \alpha(\nu) - |\alpha(\lambda) - \alpha(\nu)| \bigr)= \\
  = \frac{1}{2} \bigl( \sum_{\alpha \in \Delta_+} |\alpha(\lambda)| +
  |\alpha(\nu)| - |\alpha(\lambda-\nu)| \bigr) = \frac{1}{2} [
  \log_{q}\vol (Jt^\nu J) + \log_{q}\vol (Jt^\lambda J) - \log_{q}
  \vol (Jt^{\nu-\lambda}J) ]. \qedhere
\end{multline*}
\end{proof}

\subsubsection{Semismallness} \label{sss:semismall} Abusively, we will
let $\vol$ also denote the product Haar measure on $G(F) \times G(F)$.

\begin{proposition} \label{p:semismall} Let $\lambda$ be dominant and
  $\nu$ be antidominant coweights.  Let $p^{\lambda,\nu}:
  (Jt^{\lambda}J)\times (Jt^{\nu}J) \ra G(F)$ denote the restriction
  of the multiplication map. For every $x\in
  Jt^{\lambda+\nu}J$,\footnote{Note that according to Proposition
    \ref{p:cosetMult}.(c), the only relevant double coset in the image
    of $p^{\lambda,\nu}$ is $Jt^{\lambda+\nu} J$.}
\[
\log_{q}\vol((p^{\lambda,\nu})^{-1}(x)) = \frac{1}{2} \bigl(
\log_{q}\vol(Jt^{\lambda}J)+ \log_{q} \vol(Jt^\nu J) - \log_{q}
\vol(Jt^{\lambda+\nu} J ) \bigr).
\]
\end{proposition}

Let $P(x):=(p^{\lambda,\nu})^{-1}(x)$.  Let $\pi_{1}:
(Jt^{\lambda}J)\times (Jt^{\nu}J) \ra Jt^{\lambda}J $ denote the
projection onto the first factor. Let $\pi$ denote the restriction of
$\pi_{1}$ to $P(x)$. The proposition follows from Corollary
\ref{c:volumeCosetMult}.(iii) and the following lemma.

\begin{lemma}
  $\pi$ is injective and its image is canonically identified with
  $t^{-\lambda}Jt^{\lambda}J\cap t^{\nu} J t^{-\nu}J$.
\end{lemma} 
 
\begin{proof}  
  The fact that $\pi$ is injective is evident.  The image of $\pi$
  identifies with $y\in Jt^{\lambda}J$ such that $y^{-1}x\in
  Jt^{\nu}J$, i.e., $y\in xJt^{-\nu}J$. Write $x=jt^{\lambda+\nu}j'$,
  for $j,j'\in J$. Then, the image of $\pi$ equals
\[
Jt^{\lambda}J \cap xJt^{-\nu}J = Jt^{\lambda}J\cap
jt^{\lambda+\nu}Jt^{-\nu}J \iso Jt^{\lambda}J \cap
t^{\lambda+\nu}Jt^{-\nu}J \iso t^{-\lambda}Jt^{\lambda}J \cap
t^{\nu}Jt^{-\nu}J. \qedhere
\]
\end{proof} 


\section{Representations via compact open subgroups} \label{s:repViaCompact}

\subsection{Families of principal series representations}\label{ss:familyOfRep}
Fix a (continuous) character $\bmu:T(\cO)\ra \bQlt$.  In \S
\ref{sss:familyofRep}, we defined a family of principal series
representations $\Pi$ associated to $\bmu$. We now give an alternative
definition of this family.  Let $B^0 := U(F) T(\cO)$.  Abusively, let
$\bmu$ also denote the extension of $\bmu: T(\cO) \to \bQl^\times$ to
$B^0$ such that $\bmu|_{U(F)} = 1$. Then, it follows from the
definition that
\[
\Pi \cong \ind_{B^0}^{G(F)} \bmu := \{f: G(F) \to \bQlt \mid f(gb) =
f(g)\bmu(b), \forall \, b \in B^0\},
\]
and the action is the left regular one; i.e., $g \cdot f(x) =
f(g^{-1}x)$.

\subsubsection{Realization via compact open
  subgroups} \label{sss:compactGr} For every $\alpha \in \Delta$, let
$c_\alpha := \cond(\bmu \circ \alpha^\vee)$ denote the conductor of
$\bmu \circ \alpha^\vee$; that is, the smallest positive integer $c$
for which $\bmu(\alpha^\vee(1+\fp^c))=\{1\}$.  Let
\begin{equation} \label{e:fbmu} f_{\bmu}(\alpha) = \begin{cases}
    \lfloor c_\alpha/2 \rfloor,
    & \text{if $\alpha > 0$,} \\
    \lceil c_\alpha/2 \rceil, & \text{if $\alpha < 0$.}
    \end{cases}
\end{equation}

\begin{lemma}\cite[Lemma 3.4]{Roche98} \label{l:fmucond}
  Suppose that $2 \nmid q$ if $\Delta$ has an irreducible factor of
  the form $B_n, C_n$, or $F_4$, and $3 \nmid q$ if $\Delta$ has an
  irreducible factor of the form $G_2$.  Then, $f_{\bmu}$ satisfies
  conditions (a) and (b) of \S \ref{ss:rocheGroups}.
\end{lemma} 
The conditions in the lemma are designed so that the characteristic
does not divide a ratio of square-lengths of two roots; see
\emph{op.~cit.}  To avoid the above restrictions for certain
characters, see Remark \ref{r:mu-J}.
 
In particular, in view of Lemma \ref{l:uroche}, we obtain an
associated compact open subgroup $J=J_{f_{\bmu}}$ and the subgroup
$T_{\bmu}=T_{f_{\bmu}}$. By construction, $T_{\bmu}\subseteq \ker
\bmu$. Hence $\bmu$ defines a character of $T(\cO)/T_{\bmu}$ and so
can be lifted to a character $\mu: J\ra \bQlt$. Recall from the
introduction that we set $\sW:=\ind_J^{G(F)} \mu$.

\subsubsection{The isomorphism $\sW \cong \Pi$ and residue
  characteristic restrictions} \label{sss:residue} Theorem
\ref{t:Roche}.(ii) states that there is an isomorphism of
$G(F)$-modules $\sW\cong \Pi$, provided that the characteristic of
$\Fq$ is restricted: in particular, the characteristic should not be a
torsion prime for the Langlands dual group to any semistandard Levi
subgroup of $G$ (i.e., connected subgroup containing the maximal
torus). Additionally, in order for $J$ to be defined, we assume the
characteristic obeys the conditions of Lemma \ref{l:fmucond} (but see
Remark \ref{r:mu-J}).  In Roche's paper, to prove \cite[Theorem
4.15]{Roche98}, additionally the characteristic is further restricted
so as to obtain a nondegenerate bilinear form on the Lie algebra (in
particular, restricted to be greater than $n+1$ in the $A_n$ case, or
alternatively to have certain more technical conditions satisfied),
but this restriction can be lifted by considering elements of the dual
to the Lie algebra, as in \cite{Yu-ctsr}, and not associating to them
elements of the Lie algebra itself using a pairing; cf.~Appendix
\ref{ss:irrconn}.

Put together, for every irreducible direct factor of the root
system of the split reductive group, $\operatorname{char}(\Fq)$ should
not be one of the primes
\begin{equation}\label{t:residue}
\begin{tabular}{|c|c|}
  \hline
  Root system & Excluded primes \\
  \hline
  $B_n, C_n, D_n$ & $\{2\}$ \\ \hline
  $F_4, G_2, E_6, E_7$ & \{2,3\} \\ \hline
  $E_8$ & $\{2,3,5\}$ \\ \hline
\end{tabular}
\end{equation}
In particular, our results hold unconditionally for type $A_n$ groups,
including $\GL_N$.  Additionally, in the case that $J$ equals the
Iwahori subgroup (cf.~Remark \ref{ex:Iwahori}), then no restriction on
the characteristic is needed, in view of Proposition
\ref{p:irrconn}.(i), since all double cosets of the Iwahori contain an
element of $N(T(F))$ (cf.~\cite[\S 4]{Roche98}).

\begin{remark}\label{r:mu-J} We could
  avoid the restrictions of Lemma \ref{l:fmucond} if we impose
  restrictions on the conductors $c_\alpha$ of $\mu$; this would
  potentially allow characteristic $2$ in the $B_n$ case and
  characteristic $3$ in the $G_2$ case. The important thing is to
  ensure condition (a) of \S \ref{ss:rocheGroups}. In view of the
  proof of \cite[Lemma 3.4]{Roche98}, the problem arises where $p
  \alpha^\vee = q(\alpha+\beta)^\vee - r\beta^\vee$ for some $q,r$
  ($p$ is the characteristic). So, to ensure the condition, whenever
  $\langle \alpha, \beta^\vee \rangle = - p$ for roots $\alpha$ and
  $\beta$, we should ask that either $c_{\alpha+\beta} > 1$ or
  $c_\beta$ is odd. (Then, similarly, one gets that either $c_\beta >
  1$ or $c_{\alpha+\beta}$ is odd.) In particular, if short roots
  $\beta$ all satisfy $c_\beta \geq 2$, or if $c_\beta$ is odd for all
  short roots $\beta$, then the condition would appear to be
  satisfied.
\end{remark}

\subsubsection{Explicit (iso)morphism $\sW\to
  \Pi$} \label{sss:expliso} Next, following a suggestion of Drinfeld,
we give an explicit description of a morphism $\sW\to \Pi$.  Define
$p_0: G(F) \to \bQl$ by
\[
\begin{cases} p_0(g) =0 & \text{ if $g \notin J B^0$.} \\
  p_{0}(jb)= \mu(j)\bmu(b), & \forall \, j\in J, \,\, b \in
  B^0. \end{cases}
\]

One can show that $p_{0}$ is a well-defined $J$-invariant function in
$\Pi$ (we omit the easy proof).  It follows that $1 \mapsto p_0$ is a
homomorphism of $J$-modules $\mu \to \Res^{G(F)}_J \Pi$. By compact
Frobenius reciprocity, we obtain a morphism $\Phi:\sW\ra \Pi$.  One
can probably show that $\Phi$ is an isomorphism. We neither prove nor
use this fact; we will only use $\Phi$ in the proof of Theorem
\ref{t:clambda}, and we only need to know that it is a morphism of
$G(F)$-modules.

\subsection{Endomorphism rings}\label{ss:endomorphismRing} Henceforth,
we assume that $\bmu$ is regular.  In this setting we have an explicit
canonical isomorphism $\Psi: \K_0(\Rep(\hT))\iso \End_{G(F)}(\sW)$
\eqref{eq:SatakeReg} given as follows: (see, e.g., \cite[\S
1.9]{Roche09}): For every $\lambda \in \Lambda$, let $\Theta_\lambda$
denote the corresponding character of $\hT$. Then, the action of
$\K_0(\Rep(\hT))$ on $\sW$ is
\begin{equation}\label{e:endgfpi}
  ([\Theta_{\lambda}] f)(x)=f(xt^{-\lambda}), \quad 
  \quad f\in \Pi, x\in G(F), \lambda \in \Lambda.
\end{equation}

On the other hand, $\End_{G(F)}(\sW)$ is identified with the Hecke
algebra $\sH=\sH(G(F),J,\mu)$.  We will abusively let $\Psi$ also
denote the obtained isomorphism $\K_0(\Rep(\hT)) \iso \sH$. Let
\[ f_\lambda:\sH \ra \bQl,\quad
\quad f_\lambda(jt^\lambda j'):=\mu(jj'),\quad \forall j,j'\in J;
\quad f_\lambda|_{G(F) \setminus Jt^\lambda J} = 0.
\]
According to \cite[Theorem 4.15]{Roche98}, these functions are
well-defined and form a basis for $\sH$.  We now express $\Psi$
explicitly in terms of this basis:

\begin{theorem} \label{t:clambda} For all $\lambda \in \Lambda$,
  $\Psi([\Theta_\lambda])=b_\lambda f_\lambda$ where
  $b_\lambda=q^{-\sum_{\alpha\in \Delta_+} \max\{\alpha(\lambda),
    0\}}$.
 \end{theorem}

 \begin{remark} \label{r:clambdaVol} In view of Lemma \ref{l:vollam},
   $b_\lambda b_{-\lambda} = \vol(Jt^\lambda J)^{-1}$.
 \end{remark} 
 
 The fact that $\Psi$ sends $[\Theta_\lambda]$ to a multiple of
 $f_\lambda$ is known as ``preservation of support'' and is proved by
 Roche \cite[\S 6]{Roche98} using methods of Bushnell and Kutzko
 \cite{Bushnell98}. The computation of the scalars $b_\lambda$ appears
 to be new. The first step in the computation is to show that
 $b_\lambda^{-1} = \Phi(f_\lambda)(t^\lambda)$. Then one explicitly
 computes $\Phi(f_\lambda)$ in terms of $p_0$. For details of the
 proof, see \S \ref{ss:clambda}.


 \section{Geometrization of the vector spaces $\sH$ and
   $\sW$} \label{s:geom} For the rest of this paper, we assume that
 the character $\bmu:T(\cO)\ra \bQlt$ is regular. Our goal is to
 geometrize the vector spaces $\sH$ and $\sW$, the convolution product
 of $\sH$ and the action of $\sH$ on $\sW$ by convolution.

 We will deal with both set-theoretic and scheme-theoretic points of
 varieties (set-theoretic points of $\Spec B$ are, by definition, the
 prime ideals of $B$).  When a point is scheme-theoretic, we will
 specify it as an $R$-point for some $\Fq$-algebra $R$; otherwise, we
 will be referring to a set-theoretic point. For our conventions (and
 some recollections) regarding perverse sheaves and bounded
 $\ell$-adic derived categories see Appendix \S \ref{s:perverse}. We
 only mention here that if $f:X\ra Y$ is a morphism of algebraic
 varieties, then the pushforwards $f_!$ and $f_*$ are always derived,
 and accordingly we omit any prefix of $\mathrm{R}$ (for right
 derived).

\subsection{Recollections on the affine Grassmannian} \label{sss:Gr}
It is well known that there exists a group ind-scheme $\bG$ over $\Fq$
such that $\bG(\Fq)=G(F)$. Moreover, there exists a proalgebraic group
$\bGO$ over $\Fq$ such that $\bGO(\Fq)=G(\cO)$. The affine
Grassmannian $\bGr$ is the fpqc quotient $\bG/\bGO$. There exist
proper schemes of finite type over $\Fq$,
\[
\bGr_1\subset \bGr_2\subset \cdots,
\]
which are fixed under the action of $\bGO$ (which factors through
finite dimensional quotients of $\bGO$), and whose union is $\bGr$;
see, for instance, \cite[\S 11]{Lusscf} and \cite[Proposition
1.2.2]{Ginzburg00}. Thus, $\bGr$ is an ind-proper scheme of ind-finite
type. According to \cite[\S 5.2.1]{Gaitsgory99}, this ind-scheme may
be non-reduced. This will not affect us, however, since we are only
interested in perverse sheaves on $\bGr$ (or on related ind-schemes).

Next suppose that $\bK$ is a closed subgroup of $\bGO$ such that
$\bGO/\bK$ is finite dimensional.  Let $\bY:=\bG/\bK$ and let
$\pi: \bY \ra \bGr$ denote the canonical morphism. Then
$\bY_i:=\pi^{-1}(\bGr_i)$ is a scheme of finite type and $\bY$ is the
union of the $\bY_i$. Therefore, $\bY$ is an ind-scheme of ind-finite
type.

\subsection{Geometrization of $\sW$ and $\sH$} 
\subsubsection{Geometrization of $J$ and
  stabilizers} \label{sss:Jgeom} Let $J$ be a group of the form
defined in \S \ref{ss:rocheGroups}. There exists a proalgebraic
subgroup $\bJ< \bGO$ such that $\bJ(\Fq)=J$. Indeed, $J$ has a
combinatorial description in terms of $\fp^n$ for various $n$, and it
is easy to deduce that it is the group of points of a proalgebraic
group. Moreover, the Iwahori decomposition of $J$ implies that $\bJ$
is connected.

Alternatively, $\bJ$ can be constructed abstractly as follows: by
Bruhat-Tits theory \cite{BTgrcl1}, there exists a canonical affine
smooth group scheme $\underline{J}$ over $\cO$ such that $\uJ(\cO)=J$
(characterized by additional properties). Applying the Greenberg
functor \cite{Greenberg}, we obtain for every $n\geq 1$, a connected
algebraic group $\bJ^{(n)}$ over $\Fq$ such that
$\bJ^{(n)}(\Fq)=\uJ(\cO/\fp^n)$. It follows that $\bJ:=\varprojlim
\bJ^{(n)}$ is a proalgebraic group over $\Fq$ and
\[
\bJ(\Fq)=\varprojlim \bJ^{(n)}(\Fq)=\varprojlim \uJ(\cO/\fp^n) =
\uJ(\cO).
\] 

Next, observe that $\bJ$ acts on
$\bG$ by left and right multiplication. For every $x\in
\bG(\Fq)=G(F)$, one can consider the stabilizer $\Stab_{\bJ\times \bJ}
(x)$. This is a proalgebraic subgroup of $\bJ\times \bJ$. Projection
onto the first factor defines an isomorphism of proalgebraic groups
$\Stab_{\bJ\times \bJ}(x)\cong \bJ \cap x \bJ x^{-1}$. When $G =
\GL_N$, one can show that $\Stab_{\bJ\times \bJ}(x)$ is connected
for all $x\in \GL_N(F)$, using the fact that $\GL_N$ is the set of
invertible elements of the algebra of $N\times N$ matrices.  For
arbitrary (connected split reductive) $G$, we don't know if this
stabilizer group is connected.  However, the following will suffice
for our purposes. Let $N(T(F))$ denote the normalizer of $T(F)$.

\begin{proposition} \label{p:irrconn} 
   \begin{enumerate} 
   \item[(i)]   For all $n \in N(T(F))$,  $\Stab_{\bJ \times
    \bJ}(n)$ is
  connected.
\item[(ii)] If $x \in G(F)=\bG(\Fq)$ is not in $\bJ t^\lambda \bJ$ for
  any $\lambda \in \Lambda$, then $\Stab_{\bJ \times
    \bJ}(x)^\circ(\Fq) \nsubseteq \ker(\mu \times \mu^{-1})$.
  \end{enumerate} 
\end{proposition}
We think of (i) as saying that the stabilizer of an element of the
affine Weyl group $W_{\aff} = N(T(F))/T(F)$, considered as an element
of $G(F)$ by any section of the quotient, is connected. Note that the
stabilizer does not depend on the choice of section, because, for all
$t \in T(F)$, $\Stab_{\bJ \times \bJ}(n t) \cong \bJ \cap n t \bJ
t^{-1} n^{-1} = \bJ \cap n \bJ n^{-1}$, so the stabilizer of $nt$ is
independent of $t$ up to isomorphism.

\begin{proof} First we prove (i). Let $n \in N(T)$ have image $w \in
  W_{\aff}$ in the affine Weyl group.  Furthermore, let $w_0 \in W
  \cong W_{\aff} / \Lambda$ be the image in the finite Weyl group.
  Note that $n U_{\alpha, i} n^{-1} = U_{w(\alpha, i)}$, where
  $W_{\aff}$ acts on $\Delta \times \bZ$ by the usual action.  By
  Lemma \ref{l:uroche}.(ii), the direct product $T(\cO) \cdot
  \prod_{\alpha \in \Delta} U_{\alpha,f(\alpha)}$ equals $J$ for any
  choice of ordering of $\Delta$ such that $\Delta_+$ appears first,
  followed by $\Delta_-$.  Moreover, we can replace $\Delta_{\pm}$ by
  $w_0^{-1}(\Delta_{\pm})$, and infer that it is also acceptable to
  have $w_0^{-1}(\Delta_+)$ appear first, followed by
  $w_0^{-1}(\Delta_-)$.  Hence, $J \cap n J n^{-1} = T(\cO) \cdot
  \prod_{\alpha \in \Delta} U_{\alpha,
    \max\{f(\alpha),w(w_0^{-1}(\alpha),f(w_0^{-1}(\alpha)))_2\}}$,
  where the subscript of $2$ denotes the second component of a pair in
  $\Delta \times \bZ$, and we take an ordering where $\Delta_+$
  appears first followed by $\Delta_-$.  We conclude that this
  intersection is connected.

  Part (ii) is a strengthening of \cite[Theorem 4.15]{Roche98}, which
  can be extracted from \emph{op.~cit.} along with \cite{ARirpg} and
  \cite{Araksr}. For details of the proof, see \S \ref{ss:irrconn}.
  \end{proof} 

\subsubsection{Geometrization of $\mu$} 
Let $\bmu: T(\cO)\ra \bQlt$ be a (regular) character and let
$J=J_{\bmu}$ be the corresponding compact open subgroup of $G(\cO)$
(\S \ref{sss:compactGr}). To geometrize $\mu$ (and later $\sH$ and
$\sW$) it is convenient to define two auxiliary groups. Fix, once and
for all, a positive integer $c$ such that $\bmu|_{T_c}$ is
trivial. Define
\begin{equation}
 J' := \langle U_{f_{\bmu}}, T_c \rangle,\quad \quad A:=J/J'. 
 \end{equation} 
 For example, if $\bmu$ factors through a character
   of $T(\Fq)$, we can take $J$ to be the Iwahori group, and  $c=1$. In
   this case, $J'$ is the prounipotent radical of $J$. 
 
   Note that $J'$ is well known as the subgroup $J' = J_{f'_{\bmu}}$
   where $f'_{\bmu}: \Delta \cup \{0\} \to \bZ_{\geq 0}$ is the
   concave function defined by $f'_{\bmu}|_{\Delta} = f_{\bmu}$ and
   $f'_{\bmu}(0) = c$. According to \cite{Yu}, there exists a
   canonical smooth group scheme $\uJ'$ over $\cO$ such that
   $\uJ'(\cO) = J'$ (characterized by additional properties). As in \S
   \ref{sss:Jgeom}, using the Greenberg functor \cite{Greenberg}, we
   can construct a proalgebraic group $\bJ'$ such that
   $\bJ'(\Fq)=J'$. We now give a proof of this fact independent of the
   results of Greenberg and Yu.

   \begin{lemma} \label{l:J'algebraic} The groups $J'$ and $A$ are the
     sets of $\Fq$-points of connected proalgebraic and connected
     commutative algebraic groups $\bJ'$ and $\bA$ over $\Fq$,
     respectively.
\end{lemma}  

\begin{proof} 
  Let $\bT_{\cO} < \bG_{\cO}$ be the obvious proalgebraic subgroup
  whose $\Fq$-points is $T(\cO)$. In view of the Iwahori decomposition
  $J = J^- J^0 J^+$, we see that $J' = J^- (J' \cap J^0) J^+$, which
  is a direct product decomposition.  It suffices to show that $T' :=
  J' \cap J^0 = \langle T_c, T_{f_{\bmu}} \rangle$ is the group of
  $\Fq$-points of a proalgebraic subgroup $\bT' < \bT_{\cO}$, and that
  the quotient $\bT_{\cO} / \bT' \cong \bJ/\bJ' = \bA$ has finite type. This
  is relatively easy to see, but for the convenience of the reader
  (and independent interest), we explicitly describe $T'$ and
  $T_{f_{\bmu}}$ in \S \ref{ss:J'algebraic}.
\end{proof} 

It is clear from the assumptions that $\mu$ is trivial on $J'$. In
other words, $\mu$ is the pullback of a character $\mu_0: A\ra \bQlt$
along the canonical morphism $J\ra A$. We now apply the construction
of \S \ref{sss:multLocalSystem} to obtain a one-dimensional character
sheaf $\cM_{0}$ on $\bA$ whose trace function is $\mu_0$.  Let $\cM$
be the pullback of $\cM_{0}$ via the natural morphism $\bJ\ra
\bA$. The local system $\cM$ is our geometrization of $\mu:J\ra
\bQlt$. The following result is an immediate consequence of
Proposition \ref{p:irrconn}.(ii).
\begin{corollary} \label{c:restM} Let $x\in G(F)$ be a point which is
  not contained in any relevant double coset. Then the restriction of
  $\cM$ to $\Stab_{\bJ\times \bJ}(x)^\circ$ is nontrivial.
\end{corollary}
\begin{remark}\label{r:algclsd}  The results of this and subsequent sections
  (and in particular the main theorems in \S \ref{s:conv}) can be
  generalized in a manner that allows one to replace $\Fq$ with 
  any
  algebraically closed field of characteristic restricted as in \S
  \ref{sss:residue}.
  To do so, one must eliminate $\bmu$ and begin with $\bJ, \bJ'$, and
  $\cM_0$.  In more detail, instead of beginning with $\bmu$, one
  begins with a choice of Roche's subgroup $\bJ=\bJ_f$, cf.~\S
  \ref{ss:rocheGroups}, where $f$ is determined by coefficients
  $c_\alpha \geq 1$ as in \eqref{e:fbmu}, i.e., for all $\alpha \in
  \Delta_+$, either $f(\alpha) = f(-\alpha)$ or $f(\alpha) =
  f(-\alpha)-1$.  One must then pick a $c$ such that $c \geq c_\alpha$
  for all $\alpha$, and define the corresponding $\bJ' < \bJ$.  Next,
  one can allow $\cM_0$ to be a one-dimensional character sheaf on
  $\bA$ such that: (i) $\cM_0$ is regular: its pullback to $\bT_\cO$
  under the projection $\bT_\cO \onto \bT_\cO / \bT' \cong
  \bJ/\bJ'=\bA$ has trivial stabilizer under the Weyl group action
  (i.e., the corresponding local system on $\bT_\cO$ is not isomorphic
  to its pullback under the action of any element of the Weyl group);
  and (ii) the restriction of $\cM_0$ to the one-parameter subgroup of
  $\bT_\cO$ corresponding to each coroot $\alpha^\vee$ is nontrivial
  on $\alpha^\vee(1+\fp^{c_\alpha-1})$.  As before, $\cM$ is defined
  to be the pullback of $\cM_0$ to $\bJ$.  Provided the centralizers
  of restrictions of $\cM_0$ as in the proof of Proposition
  \ref{p:irrconn}.(ii) (Appendix \ref{ss:irrconn}) remain semisimple
  (which will be true, for instance, if $\cM_0$ is obtained from one
  of the $\cM_0$ in the case of $\Fq$ by base change), all the
  statements and proofs go through with this generalization.
  We note that, in the algebraically closed setting, the
  Tate twists should be suppressed.
\end{remark}
\begin{remark}\label{r:charzero}
  One can also consider analogues of our results over an
  algebraically closed field of characteristic zero. However, since
  the affine line is simply connected in this case, $\bbG_a$ admits no nontrivial character sheaves. Thus, the
  restriction of $\cM_0$ to $\alpha^\vee(1+\fp^{c_\alpha-1})$ is
  always trivial if $c_\alpha > 1$, and we are reduced to the case
  $c_\alpha = 1$ for all $\alpha$, with $J$ the Iwahori subgroup.
\end{remark}
\begin{remark}\label{r:anycoeff} If we work over $\Spec \bC$, we can
  also consider sheaves in the complex topology, and then, as in \cite{MV},
  we can use coefficients in an arbitrary commutative Noetherian ring
  of finite global dimension.
In the case that $R$ is a field,
all of the results and proofs here go through without modification
(see, e.g., \cite{Dimca} for the necessary facts about perverse
sheaves in this context), with  $c_\alpha = 1$ for all $\alpha$, in
accordance with the previous remark.  We believe (but have not
carefully checked) that our main results also extend to the case of
commutative Noetherian rings of finite global dimension, provided one
replaces the mention of simple or irreducible objects (e.g., Corollary
\ref{c:IrredPerv}) by the statement that all objects of $\sHg$ are
finite direct sums of objects of the form $j_!^\lambda \otimes_R L$,
for $L$ a finitely-generated $R$-module.
 \end{remark}

 \subsubsection{Geometrization of $\sW$ and $\sH$} Recall that $\sW$
 is defined to be the vector space of functions on $G(F)$ which
 satisfy $f(gj)=f(g)\mu(j)$ for all $g\in G(F)$ and $j\in J$.
 Equivalently, $\sW$ is the vector space of functions on $G(F)/J'$
 satisfying $f(ga)=f(g)\mu_0(a)$ for all $g\in G(F)/J'$ and $a\in
 A$. Similarly, $\sH$ is the vector space of functions on $G(F)/J'$
 satisfying $f(jga)=\mu_0(j)f(g)\mu_0(a)$ for all $j \in J$ and $a \in
 A$.  Using this observation, geometrizing $\sW$ and $\sH$ becomes
 straightforward.

Let $\bX:=\bG/\bJ'$. By the discussion in \S \ref{sss:Gr}, $\bX$ is a
union of schemes $\bX_i$ of finite type over $\Fq$.
The bounded constructible derived category $\sD(\bX)$ of sheaves on
$\bX$ is, by definition, the inductive limit of $\sD(\bX_i)$ (see
Appendix \S \ref{s:perverse} for our conventions regarding the derived
category and perverse sheaves).

The connected algebraic group $\bA$
acts freely on $\bX$ by the right multiplication action $r:\bA \times
\bX \ra \bX$, $r(a,x) := xa^{-1}$. Similarly, $\bJ$ acts on
$\bX$ by the left multiplication action
 $l:\bJ\times \bX\ra \bX$. 
The scheme $\bX_i$ is invariant under the action of
$\bJ\times \bA$. Indeed,
\[
\pi(\bJ \bX_i \bA)\subseteq \pi(\bG_\cO \bX_i \bA) \subseteq \bG_\cO
\bGr_i = \bGr_i.
\]
The left action of $\bJ$ on each $\bX_i$ clearly factors through a
finite dimensional quotient. Let $\bJ_i$ be such a quotient for each
$i$ which factors the quotient $\bJ \onto \bA$, and let $l_i: \bJ_i
\times \bX_i \to \bX_i$ be the resulting map descending from $l$.
Furthermore, let $\cM^i$ be the pullback of $\cM_0$ to $\bJ_i$ under
the quotient $\bJ_i \onto \bA$. Each $\cM^i$ is a multiplicative local
system on $\bJ_i$. Let $\sWg^i$ be the full subcategory of perverse
sheaves on $\bX_i$ satisfying $r^* \cF \cong \cM_0^{-1} \boxtimes
\cF$, and let $\sHg^i$ be the full subcategory of perverse sheaves on
$\bX_i$ satisfying $(l_i \times r)^* \cF \cong \cM^i \boxtimes
\cM_0^{-1} \boxtimes \cF$. Let $\sWg$ denote the direct limit of the
abelian categories $\sWg^i$, and let $\sHg$ denote the direct limit of
the abelian categories $\sHg^i$. We consider the objects of $\sWg$ the
$(\bA, \cM_0^{-1})$-equivariant perverse sheaves on $\bX$, and the
objects of $\sHg$ the $(\bJ \times \bA, \cM \times
\cM_0^{-1})$-equivariant perverse sheaves on $\bX$.

Taking trace of Frobenius of elements of the abelian categories $\sWg$
and $\sHg$, we recover the vector spaces $\sW$ and $\sH$.


\subsection{Objects of $\sHg$} 

\subsubsection{Relevant orbits}
Recall that we called double cosets of the form $Jt^\lambda J$
relevant. We call the schemes $\bJ^\lambda := \bJ (t^\lambda \bJ')
\bA\subseteq \bX$ \emph{relevant orbits}. All the facts that we proved
about relevant double cosets in \S \ref{ss:relevantCosets} hold for
relevant orbits. This is because our arguments, which concerned
$\Fq$-points, carry over to $R$-points for any $\Fq$-algebra $R$.
Similarly, the facts about volume proved in \S \ref{ss:volume} compute
the dimension of the associated schemes, under the correspondence
\[
\log_q(\vol Y) = \dim \bY - \dim \bA,
\]
 where $Y$ is one of the
subsets obtained from cosets used in \S \ref{ss:volume}, and $\bY
\subseteq \bX$ is the associated subscheme of $\bX$ satisfying
$\bY(\Fq) = Y/J'$.

Below, we will often use the following basic fact: each stratum
$\bX_i$ contains only finitely many relevant orbits.  This is true
because $\bGr_i$ contains only finitely many cosets of the form
$t^\lambda \bG_{\cO}$ for $\lambda \in \Lambda$ (which follows, for
example, by taking a standard choice of $\bGr_i$; see, e.g.,
\cite{Ginzburg00}).

\subsubsection{Geometrization of $f_\lambda$}
Recall that $f_{\lambda}:Jt^{\lambda}J\ra \bQl$ is defined by
$f_\lambda(jt^{\lambda}j')=\mu(j)\mu(j')$. Our goal is to geometrize
the following statements:
 
\begin{enumerate}
\item[(i)] The restriction of $\mu \times \mu^{-1}$ to $\Stab_{J \times
    J}(x)$ is trivial for all $x \in Jt^\lambda J$;
\item[(ii)] $f_\lambda$ is the unique function on $Jt^\lambda J$ whose
  pullback to $J\times J$ under the map $(j_1,j_2) \mapsto j_1
  t^\lambda j_2$ equals $\mu\times \mu^{-1}$;

\item[(iii)] $f_{\lambda}$ is the unique function on $Jt^\lambda J$,
  up to a scalar multiple, which is $(J\times J, \mu\times
  \mu^{-1})$-invariant.
\end{enumerate} 
 Here and below, $J\times J$ on $G(F)$ acts by left and
  right multiplication, i.e., $(j_1, j_2) \cdot g = j_1 g j_2^{-1}$.

\begin{lemma} \label{l:Flambda}
\begin{enumerate}
\item[(i)] For every $\lambda\in \Lambda$ and $x \in \bJ^\lambda$, the
  restriction of $\cM \boxtimes \cM_0^{-1}$ to ${\Stab_{\bJ \times
      \bA}(x)}$ is trivial.
         
\item[(ii)] There exists a unique, up to isomorphism, local system
  $\cF_{\lambda}'$ on $\bJ^{\lambda}$ such that $(\pi^{\lambda})^{*}
  \cF_{\lambda}'\cong \cM\boxtimes \cM_{0}^{-1}$, where $\pi^\lambda:
  \bJ\times \bA \ra \bJ^\lambda$ denotes the map $(j,a) \mapsto
  jt^\lambda a^{-1} \bJ'$.
\item[(iii)] Suppose $\cG$ is a $(\bJ\times \bA, \cM \boxtimes
  \cM_0^{-1})$-equivariant local system on $\bJ^\lambda$. Then $\cG\cong
  \cF_\lambda' \otimes \cL$ where $\cL$ is the pullback, via
  $\bJ^\lambda\ra \Spec(\Fq)$, of a local system on $\Spec(\Fq)$.
\end{enumerate} 
\end{lemma} 
\begin{proof}
  Note that $\cM \boxtimes \cM_0^{-1}$ is pulled back from $\cM_0
  \boxtimes \cM_0^{-1}$ on $\bA \times \bA$.  To prove (i), it is
  enough to show that the restriction of $\cM_0 \boxtimes \cM_0^{-1}$
  to the image of $\Stab_{\bJ \times \bA}(t^\lambda \bJ')$ in $\bA
  \times \bA$ is trivial.  The latter is contained in the diagonal,
  $\{(a,a) \mid a \in \bA\} \subseteq \bA \times \bA$, and the
  restriction of $\cM_0 \boxtimes \cM_0^{-1}$ to this locus is a
  tensor product of two inverse local systems, which is trivial.

  (ii) follows immediately from (i) by equivariant descent, since
  $\pi^\lambda$ is the quotient by the (free) action of $\Stab_{\bJ
    \times \bA}(t^\lambda)$.
  
  For (iii) consider the local system $\cL := (\cF_\lambda')^{-1}
  \otimes \cG$ on $\bJ^\lambda$.  Then, $\cL$ is a $(\bJ \times
  \bA)$-equivariant local system on $\bJ^\lambda$.  Hence, $(l \times
  r)^*(\cL)$ is a local system on $(\bJ \times \bA) \times
  \bJ^\lambda$ which is trivial in the $(\bJ \times \bA)$ direction.
  If we restrict to $(\bJ \times \bA) \times \{t^\lambda \bJ'\}$, we
  obtain that $(\pi^\lambda)^*(\cL)$ (with $\pi^\lambda$ as in (ii))
  is pulled back from the local system $\cL|_{t^\lambda \bJ'}$ on
  $t^\lambda \bJ' \cong \Spec \Fq$.  Thus, $\cL$ is also pulled back
  from a local system on $\Spec \Fq$.
  \end{proof}

  \subsubsection{The local systems $\cF_\lambda$ and their
    extensions} \label{sss:flambda} Let $\cF_\lambda :=
  \cF_\lambda'[\dim \bJ^\lambda](-\log_q b_\lambda)$, where $(m)$
  denotes the Tate twist by $m$.  Then $\cF_{\lambda}$ is a $\bJ\times
  \bA$-equivariant perverse sheaf on $\bJ^{\lambda}$ and 
\begin{equation}\label{e:trcfl}
\Tr(\Fq,  \cF_{\lambda})=(-1)^{\dim \bJ^\lambda}b_{\lambda}f_{\lambda}.
\end{equation}
 (see
  \S \ref{ss:traceFrob} for our conventions regarding the trace of
  Frobenius).  Let $j^\lambda: \bJ^\lambda \into
\bX$ denote the natural inclusion. To these are associated
derived functors $j_!^\lambda$ and $j_*^\lambda$ from
$\sD(\bJ^\lambda)$ to $\sD(\bX)$. We use the
abusive abbreviations
\begin{equation}
  j^\lambda_! := j^\lambda_! \cF_\lambda, \quad \quad j_{!*}^{\lambda}:= 
  j_{!*}^{\lambda}\cF_{\lambda}, \quad \quad
  j_{*}^{\lambda}:= j_{*}^{\lambda}\cF_{\lambda};
\end{equation}

We will eventually see that $j_!^\lambda \cong j_{!*}^\lambda \cong
j_*^\lambda$ (Theorem \ref{t:clean}).

\begin{lemma} \label{l:j!*Perv} $j_!^\lambda$, $j_{!*}^\lambda$, and
  $j_*^\lambda$ belong to $\sHg$.
\end{lemma}

\begin{proof} By definition, $j_{!*}^\lambda$ is a perverse sheaf.
  Since $\bJ \times \bA$ is solvable, all of its orbits are
  affine. Thus, $j^\lambda$ is an open affine embedding. Therefore,
  $j_!^\lambda$ and $j_*^\lambda$ are perverse sheaves as well. To
  prove the result, it remains to show that these perverse sheaves are
  $(\bJ\times \bA, \cM\boxtimes \cM_0^{-1})$-equivariant.  Using the
  projection formula (twice),
\[
l^* (j^\lambda_! \cF_\lambda) \cong (\Id\times j^\lambda)_! (l^*
\cF_\lambda) \cong (\Id\times j^\lambda)_! (\cM \boxtimes \cF) \cong
\cM \boxtimes j^\lambda_!\cF_\lambda.
\]
Thus, $j_!^\lambda$ is $(\bJ,\cM)$-equivariant. Similarly, one shows
that $j_*^\lambda$ is $(\bJ,\cM)$-equivariant for the left
multiplication action. Now $j^\lambda_{!*}$ is the image of the
canonical morphism $j^\lambda_!\ra j^\lambda_*$. Since the canonical
morphism is functorial, it is easy to see that $l^*(j_{!*}^\lambda)
\cong \cM \boxtimes j_{!*}^\lambda$. One proves in an analogous manner
that $j^\lambda_!, j^\lambda_*$ and $j^\lambda_{!*}$ are $(\bA,
\cM_0^{-1})$-equivariant for the right multiplication action.
\end{proof}

\subsubsection{Restriction to irrelevant points}
Let an \emph{irrelevant point} $x \in \bX$ denote a point which does
not lie in any relevant orbit. Recall that if $f\in \sH$, then
$f(x)=0$ for all irrelevant points $x$. In this section, we prove a
geometric analogue of this statement.

\begin{proposition} \label{p:restrictionIrrelevant} Let
  $y \in \bX$ be an irrelevant (set-theoretic) point.  
  \begin{enumerate} 
  \item[(i)] The stalk of every $\cF\in \sHg$ at $y$ is zero.
  \item[(ii)] The stalk at $y$ of every $(\bJ\times \bA, \cM\boxtimes
    \cM_0^{-1})$-equivariant complex $\cG \in \sWgd$ is zero.
  \end{enumerate} 
\end{proposition}

\begin{proof} It is clear that (i) is a consequence of (ii), so we
  only prove (ii).

First, we claim that it suffices to assume that $y$ is closed.
Indeed, otherwise, since finitely many relevant orbits lie in each
stratum $\bX_i$, only finitely many can intersect the closure
$\bar{y}$ of $y$. Now the complement of these relevant orbits in
$\bar{y}$ is a dense open subvariety $U$ of $\bar{y}$. If the stalks
at the (necessarily irrelevant) closed points in $U$ vanish, then the
restriction of $\cG$ to $U$ vanishes.  Hence, the restriction of $\cG$
to $y \in U$ also vanishes.

So, assume that $y$ is closed.  By Lemma \ref{l:restrictionTrivial},
it suffices to show that $\cM \boxtimes \cM_0^{-1}$ is nontrivial on
the connected component of the identity of the stabilizer of $y$; then
it would follow that all cohomology sheaves of $\cG$ have zero stalk
at $y$, and hence also that the stalk of $\cG$ at $y$ is zero.  For
$\Fq$-points the result follows from Corollary \ref{c:restM}. For
$\Fqn$-points one can use the norm maps (Remark \ref{r:norm}).  In
more detail, we can replace $\bmu$ by the corresponding character of
$T(\Fqn\dblsqbrs{t})$ and work over the field $\Fqn$.  This implies
the result for all set-theoretic points $y$.
\end{proof}

\begin{corollary}\label{c:IrredPerv} The irreducible objects in $\sHg$
  are of the form $j^\lambda_{!*} \otimes \cL$,
where $\cL$ is the pullback, via $\bJ^\lambda\ra \Spec \Fq$, of a
one-dimensional local system on $\Spec \Fq$.
\end{corollary} 
\begin{proof}
  Note that $j_{!*}^\lambda$ is irreducible for every $\lambda \in
  \Lambda$, since $\cF_\lambda$ is irreducible (in fact,
  one-dimensional).  For the converse, let $\cF$ be an irreducible
  object of $\sHg$. Then there must exist a $\bJ \times \bA$-invariant
  locally closed subscheme $\bY \subset \bX_i$ of one of the strata
  $\bX_i$ of $\bX$, and an irreducible $(\bJ \times \bA, \cM \boxtimes
  \cM_0^{-1})$-equivariant local system $\cG$ on $\bY$ such that $\cF
  = j_{!*}^{\bY} \cG[\dim \bY]$, where $j^{\bY}: \bY \into \bX$ is the
  inclusion.  Hence, $\bY$ must lie in the union of the finitely many
  relevant orbits in $\bX_i$. Since it is $\bJ \times \bA$-invariant,
  $\bY$ equals a finite union of relevant orbits.  As $\cF$ is
  irreducible, $\bY$ is also irreducible.  Therefore, there must exist
  $\lambda \in \Lambda$ such that $\bJ^{\lambda} \cap \bY$ is open and
  dense in $\bY$.  Since $\bY$ is $\bJ \times \bA$-invariant, in fact
  $\bJ^{\lambda} \subseteq \bY$.  Thus, we conclude that $\cF \cong
  j^\lambda_{!*} (\cF|_{\bJ^{\lambda}})$.  Hence, Lemma
  \ref{l:Flambda} implies that it has the desired form (note that all
  irreducible local systems on $\Spec \Fq$ are one-dimensional).
\end{proof}


\section{Convolution product and main results} \label{s:conv} 
\subsection{Definition of convolution} \label{ss:defConv} Let
$p:\bG\times_{\bJ} \bX\ra \bX$ denote the product map (where
$\times_{\bJ}$, as in the setting of $\Fq$-points in \S
\ref{ss:relevantCosets}, denotes the quotient of the product by the
inner adjoint action of $\bJ$, $j \cdot (g, x) = (gj^{-1}, jx)$).  The
convolution with compact support is the functor defined by
\begin{equation}\label{eq:conv}
  \star_!: \sWg \times \sHg \ra \sWgd,\quad \quad (\cF,\cG)\mapsto 
  p_!(\cF\tboxtimes \cG). 
\end{equation}
Here $\cF\tboxtimes \cG$ is the the twisted external product of
twisted equivariant sheaves (\S \ref{ss:twistedProd}). We usually
write $\star$ for $\star_!$.  There are associativity isomorphisms
\begin{equation} \label{eq:associativity} \cF\star (\cG\star \cG')
  \iso (\cF\star \cG)\star \cG', \quad \quad \forall \, \cG,\cG'\in
  \sHg,\quad \cF\in \sWgd \end{equation} satisfying natural
properties; see, for instance, \cite[\S 7]{BD} or
\cite{Gaitsgory01}. One can easily check that
 \begin{equation}\label{eq:traceConv}
   \Tr(\Fr_{q^n}, \cF) \star \Tr(\Fr_{q^n}, \cG) = (-1)^{\dim \bA} 
   \Tr(\Fr_{q^n}, \cF \star \cG), \quad \quad \forall \, \cF\in \sWg,\, 
   \cG\in \sHg. 
  \end{equation}
  Thus, up to sign, this geometrizes the usual
  convolution product $\sW \star \sH\ra \sW$ (\ref{eq:conv}).

\subsection{Convolution of $j_!^\lambda$} 
Using Lemma \ref{l:vollam}, Theorem \ref{t:clambda}, \eqref{e:trcfl},
and (\ref{eq:traceConv}), one can easily show that
 \begin{equation}\label{eq:trConv}
 \Tr(\Fr_{q^n}, j_!^\lambda \star j_!^\nu)= \Tr(\Fr_{q^n}, j_!^{\lambda+\nu}). 
\end{equation} 
The following is the geometric analogue of (\ref{eq:trConv}). It is
the key result for proving our main theorems, and was suggested to us
by D. Gaitsgory.

\begin{proposition} \label{p:conv} For all $\lambda, \nu\in \Lambda$,
  $j_!^\lambda \star j_!^\nu \cong j_!^{\lambda+\nu}$.
\end{proposition} 

In the case that $\lambda$ and $\nu$ are both dominant or
antidominant, the proposition follows easily from the isomorphism
$\bJ^\lambda \times_\bJ \bJ^\nu \cong \bJ^{\lambda+\nu}$ (Corollary
\ref{c:proddom}).  In the case that $\lambda$ is dominant and $\nu$ is
antidominant, we combine the fact that the only relevant orbit in the
closure of $\bJ^\lambda \times_\bJ \bJ^\nu$ is $\bJ^{\lambda+\nu}$
(Proposition \ref{p:cosetMult}.(c)) and the semismallness result
proved in \S \ref{sss:semismall} to show that $j_!^\lambda \star
j_!^\nu$ is perverse. It is then easy to show that it must be
isomorphic to $j_!^{\lambda+\nu}$. For details of the proof see \S
\ref{ss:conv}.

\begin{corollary}\label{c:BMWExt} For all $\lambda,\nu\in \Lambda$ and all
  local systems $\mathcal{L}, \mathcal{K}$ on $\Spec \Fq$,
\[
\Ext^\bullet(j_{!}^{\lambda} \otimes \cL , j_{!}^{\nu} \otimes
\cK)=\begin{cases} \Ext_{\Spec \Fq}^\bullet (\cL, \cK), & \text{if
    $\lambda = \nu$}, \\ 0, & \text{otherwise}.
\end{cases}
\]
 \end{corollary}

 The above corollary is proved using the monoidal property of
 $j_!^\lambda$ established in Proposition \ref{p:conv}; see \S
 \ref{ss:BMWExt} for details.

\subsection{Proof of Theorem \ref{t:main}: 
cleanness of irreducible objects} \label{ss:clean}
The following theorem clearly implies Theorem \ref{t:main}:
\begin{theorem}\label{t:clean}
  For every $\lambda\in \Lambda$, $j_!^\lambda\cong j_{!*}^\lambda$.
\end{theorem} 

\begin{proof} 
  By Corollary \ref{c:IrredPerv}, all objects are obtained by iterated
  extensions of objects of the form $j_{!*}^\lambda \otimes \cL$, for
  $\cL$ a local system on $\Spec \Fq$.  We first claim that, for fixed
  $\lambda$, $j_!^\lambda$ is obtained by iterated extensions of
  objects of the form $j_{!*}^\lambda \otimes \cL$ (the point is that
  we use the \emph{same} $\lambda$).  Inductively, it suffices to show
  that, if $j_{!*}^\nu \otimes \cL \into j_!^\lambda$ is an injection,
  then $\nu = \lambda$.  If we precompose the inclusion $j_{!*}^\nu
  \otimes \cL \into j_!^\lambda$ with the defining surjection $j_!^\nu
  \otimes \cL \onto j_{!*}^\nu \otimes \cL$, one obtains a nonzero map
  $j_!^\nu \otimes \cL \to j_!^\lambda$. By Corollary \ref{c:BMWExt},
  this implies $\nu = \lambda$.

  Now, suppose that $j_{!*}^\lambda \otimes \cL \into j_!^\lambda$ is
  an injection.  Applying Corollary \ref{c:BMWExt} again, the
  composition $j_!^\lambda \otimes \cL \onto j_{!*}^\lambda \otimes
  \cL \into j_!^\lambda$ is obtained from a map $\cL \to \bQl$. This
  map is injective, since the map $\cL \cong j_!^\lambda \otimes
  \cL|_{t^\lambda \bJ'} \cong j_{!*}^\lambda \otimes \cL|_{t^\lambda
    \bJ'} \to j_!^\lambda|_{t^\lambda \bJ'}$ is injective.  Hence the
  canonical map $j_!^\lambda \otimes \cL \onto j_{!*}^\lambda \otimes
  \cL$ is an isomorphism.  Inductively, $j_!^\lambda$ is obtained by
  iterated extensions of objects of the form $j_!^\lambda \otimes \cL$
  by maps of the form $\Id \otimes \phi$ for $\phi$ a map of local
  systems over $\Spec \Fq$.  Moreover, every object $j_!^\lambda
  \otimes \cL$ that appears is isomorphic to $j_{!*}^\lambda \otimes
  \cL$ by the canonical map. We conclude that $j_!^\lambda \cong
  j_{!*}^\lambda$ (by the canonical map).  \qedhere
\end{proof}

\subsection{Monoidal equivalence} \label{ss:monEquiv}
The following theorem clearly implies Theorem \ref{t:main2}:
\begin{theorem} \label{t:monEquiv} The category $\sHg$ is closed under
  convolution; moreover, the functor
\[
\Psi_{\geom}: \Rep \hT \boxtimes \LocSys(\Spec \Fq) \ra \sHg,\quad
\quad \Theta_\lambda \otimes \mathcal{L} \mapsto j_{!}^\lambda \otimes
\mathcal{L}
\]
is an equivalence of monoidal abelian categories.
\end{theorem}

\begin{proof} 
  Corollary \ref{c:BMWExt} and Theorem \ref{t:clean}, together with
  Corollary \ref{c:IrredPerv}, imply that all objects of $\sHg$ are
  finite direct sums of objects of the form $j_!^\lambda \otimes \cL$,
  where $\cL$ is a local system on $\Spec \Fq$. Therefore, to prove
  $\sHg$ is closed under convolution, it is enough to show that, for
  all coweights $\lambda$ and $\nu$ and all local systems $\cL$ and
  $\cL'$ on $\Spec \Fq$, $(j_{!}^{\lambda}\otimes \cL) \star
  (j_{!}^{\nu} \otimes \cL') \in \sHg$. This follows at once from
  Proposition \ref{p:conv}. Proposition \ref{p:conv} also implies that
  $\Psi$ is monoidal. In view of Corollary \ref{c:BMWExt}, we obtain
  the desired equivalence of abelian categories.
\end{proof}

\subsection{Convolutions with and without compact support are isomorphic}  
Recall that we have two, a priori different, monoidal actions of
$\sHg$ on $\sWgd$ given by $\star=\star_!$ and $\star_*$.  We now
prove that these two actions are isomorphic by proving that their
adjoints are isomorphic.

\begin{proposition} \label{p:adjoint} For every $\cF, \cG\in \sWgd$,
  there is a canonical functorial isomorphism
  \[
\Hom_{\sWgd} (\cF\star_! j_!^\lambda, \cG) \cong \Hom_{\sWgd} (\cF, \cG\star_*
j_*^{-\lambda}).
\]
\end{proposition} 

The above proposition is probably known, at least in the untwisted setting. For completeness, we include a proof in \S \ref{ss:adjoint}. 

\begin{corollary} \label{c:!=*} There exists an isomorphism between
  the actions of $\sHg$ on $\sWgd$ given by $\star_!$ and $\star_*$.
\end{corollary} 

\begin{proof} By Theorem \ref{t:monEquiv} $\sHg$ is a rigid category
  (in fact, it is a Picard category). Therefore, the adjoint functor
  to $-\star_! j_!^\lambda$ is isomorphic to $- \star_!
  j_!^{-\lambda}$. On the other hand, by the above proposition, this
  adjoint functor is also isomorphic to $-\star_* j_*^{-\lambda}$. We
  conclude that the functors $-\star_! j_!^\lambda$ and $-\star_*
  j_*^\lambda$ are isomorphic. Using cleanness (Theorem
  \ref{t:clean}), we conclude that the functors $-\star_! j_!^\lambda$
  and $-\star_* j_!^\lambda$ are isomorphic. For any $\mathcal{L} \in
  \LocSys(\Spec \Fq)$, it follows also that $-\star_! (j_!^\lambda
  \otimes \mathcal{L})$ is isomorphic to $-\star_* (j_!^\lambda
  \otimes \mathcal{L})$. The result then follows from Theorem
  \ref{t:monEquiv}.
 \end{proof} 

 As explained in the next subsection, one can (probably) show that the
 canonical morphism $\star_!\ra \star_*$ (obtained from the general
 functorial maps $f_! \to f_*$, which we call ``forgetting compact
 support'') is an isomorphism between the two actions of $\sHg$ on
 $\sWgd$. (However, we do not need this fact). We note that the idea
 of showing $\star_!$ and $\star_*$ are isomorphic by proving that
 their adjoints are isomorphic was suggested to us by
 D. Gaitsgory. The same idea is employed in \cite[\S G]{BDOrbit} and
 \cite[\S 6.7]{BD08}

 \subsubsection{Aside: relationship to Grothendieck-Verdier
   duality} \label{sss:rcategory} In \cite[Appendix A]{BD11} and
 \cite{BD11b}, Boyarchenko and Drinfeld explain the following picture.
 A monoidal category $(\cC,\otimes_1, \bone)$ is a called an
 \emph{r-category} if for every $Y\in \cC$ the functor
 $\Hom(-\otimes_1 Y, \bone)$ is representable by some object $DY$ and
 the contravariant functor $D:\cC\ra \cC$ is an antiequivalence. $D$
 is called the \emph{duality functor}. In every r-category, there is a
 second tensor product defined by
 \[
X\otimes_2 Y := D^{-1}(DY\otimes_1 DX).
\]
Moreover, there is a monoidal natural transformation $X\otimes_1 Y\ra
X\otimes_2 Y$, which is an isomorphism if and only if $\cC$ is rigid.

As an example, let $G$ be a connected algebraic group over a field
$k$. Let $\iota:G\ra G$ be the map $\iota(g)=g^{-1}$. Let
$\bbD:\sD(G)\ra \sD(G)$ denote the Verdier duality functor and
$\obD:\sD(G)\ra \sD(G)$ denote $\bbD \circ \iota^*=\iota^* \circ
\bbD$. Then $(\sD(G),\star_!)$ is an r-category with duality functor
$\obD$. The second tensor product is convolution without compact
support, $\star_*$. According to \cite[Appendix A]{BD11}, the natural
transformation $\star_!\ra \star_*$ defined above should coincide with
the canonical map of ``forgetting compact support.''

Now we apply the above considerations to our situation. The argument
in \cite[Appendix A]{BD11} for proving $(\sD(G), \star_!, \obD)$ is an
r-category applies verbatim to show that $(\sHg, \star_!, \obD)$ is an
r-category. (This amounts to a special case of Proposition
\ref{p:adjoint} where one takes $\cF,\cG\in \sHg$.) The second tensor
product in $\sHg$ is $\star_*$. In analogy with $\sD(G)$, the natural
transformation $\star_!\ra \star_*$ should coincide with the canonical
map coming from forgetting the support (but we have not checked
this). As $\sHg$ is rigid, we see that the canonical morphism
$\star_!\ra \star_*$ is an isomorphism.

Next, we have two monoidal actions of $(\sHg,\star_!)$ on $\sWgd$: one
given by $\star_!$ and the other one given by $\star_*$, via the
monoidal equivalence $(\sHg, \star_!)\iso (\sHg, \star_*)$. One can
show that the canonical ``forgetting compact support'' maps define a
monoidal natural transformation going from the first action to the
second one. Moreover, it follows from the following general lemma that
this is an isomorphism.

\begin{lemma}\label{l:rigmon}
  Let $\mathcal{C}$ be a rigid monoidal category with unit object
  $\bone$, $\mathcal{D}$ a monoidal category, $F, G: \mathcal{C} \to
  \mathcal{D}$ two monoidal functors, and $\eta \in \Hom(F,G)$ a
  monoidal natural transformation such that $\eta_{\bone}: F(\bone)
  \to G(\bone)$ is an isomorphism. Then, $\eta$ is an isomorphism.
\end{lemma}

A version of the above lemma appears in \cite[Proposition
5.2.3]{SRct}.

\subsection{Action of $\sHg$ on $\sWg$} \label{ss:proofAction} The
following implies Theorem \ref{t:action}.
\begin{theorem} Let $\cF\in \sHg$ and $\cG\in \sWg$. Then $\cG\star_!
  \cF \in \sWg$.
\end{theorem}

\begin{proof} 
  It is enough to show that $\cG\star_! j_!^\lambda$ is perverse for
  every $\lambda \in \Lambda$. Let $p^\lambda = p\circ (\Id \times
  j^\lambda)$.  Then $p_!^\lambda (\cG\tboxtimes \cF_\lambda) \cong
  \cG\star_! j_!^\lambda$.  Since $p^\lambda$ is an affine morphism,
  by Artin's Theorem \ref{t:Artin}, $\cG\star_! j_!^\lambda\in
  \sDg_{(\bA, \cM_0^{-1})}(\bX)$. By Corollary \ref{c:!=*} and
  cleanness, we obtain an isomorphism $\cG\star_! j_!^\lambda\cong \cG
  \star_* j_*^\lambda$. Applying Artin's Theorem again, $\cG\star_*
  j_*^\lambda \in \sDl_{(\bA,\cM_0^{-1})} (\bX)$. Hence, $\cG \star_!
  j_!^\lambda \in \sDl_{(\bA,\cM_0^{-1})} (\bX)$ as well.
\end{proof}

\appendix

\section{Postponed proofs} 
\subsection{Proof of Theorem \ref{t:clambda}}\label{ss:clambda}
First we describe the morphism $\Phi: \sW\ra \Pi$ explicitly. Let $f_0
\in \sW$ be the function $f_0(j) = \mu(j)$ for $j \in J$ and $f_0(g) =
0$ for $g \notin J$. Then $\sW$ has a basis consisting of functions
$g_{i} \cdot f_{0}$ where $g_{i}\in G(F)$ ranges over a set of
representatives for $G(F)/J$; see, e.g.  \cite[\S
1.2.5]{Bushnell06}. Define a morphism of $G(F)$-modules $\Phi:\sW\ra
\Pi$ by $f_{0}\mapsto p_{0}$, where $p_0$ was defined in \S
\ref{sss:expliso}.  Next, suppose that $\Omega:\sW\ra \Pi$ is any
morphism of $G(F)$-modules. Using the fact that $\End_{G(F)}(\Pi)$ is
commutative, one can easily show that
\[
\Omega( f\star \Psi(\phi))=\phi(\Omega(f)), \quad \quad \forall f\in
\sW, \phi \in \End_{G(F)}(\Pi).
\]

\begin{lemma} $b_\lambda\Phi(f_\lambda)(t^\lambda)=1$. \label{l:clambda1}
\end{lemma}
\begin{proof}
  By the above, $\Phi( f \star (b_\lambda f_\lambda) ) =
  [\Theta_\lambda] (\Phi(f))$ for all $f \in \sW$. Hence,
\[
\Phi(b_\lambda f_\lambda) = \Phi(f_0 \star (b_\lambda f_\lambda)) =
[\Theta_\lambda] (\Phi(f_0)) = [\Theta_\lambda] (p_0).
\]
The result follows by evaluating both sides at $t^\lambda$. 
\end{proof} 

Hence, to compute $b_\lambda$, it suffices to compute
$\Phi(f_\lambda)(t^\lambda)$. First we need two lemmas.

\begin{lemma} For all $j \in J$,
  \begin{equation}\label{e:p0cl}
    p_0(t^\lambda j t^{-\lambda}) = \begin{cases} \mu(j), 
      & \text{if $t^\lambda j t^{-\lambda} \in JB^0$}, \\
      0, & \text{otherwise}.
\end{cases}
\end{equation}
\end{lemma}

\begin{proof}
  Write $j = j^- j^0 j^+$, with $j^- \in J^-, j^0 \in J^0$, and $j^+
  \in J^+$.  Then, $t^\lambda j t^{-\lambda} = (t^\lambda j^-
  t^{-\lambda}) (t^\lambda j^0 j^+ t^{-\lambda})$.  So, first of all,
  \begin{equation}\label{e:tjjm}
t^\lambda j t^{-\lambda} \in J B^0 \iff t^\lambda j^- t^{-\lambda} \in J^-,
\end{equation} 
since we know that $t^\lambda j^- t^{-\lambda} \in U^-(F)$, the
group of unipotent lower-triangular matrices.  So, we find that
\begin{equation*}
  p_0(t^\lambda j t^{-\lambda}) = \mu(j^0) p_0(t^\lambda j^- t^{-\lambda}),
\end{equation*}
which yields \eqref{e:p0cl}.  
\end{proof}

\begin{lemma} \label{l:phifl} $\Phi(f_\lambda) = \sum_i \mu(j_i) j_i
  t^{\lambda} \cdot p_0$, where $j_i$ is a set of representatives of the
  finite quotient $J / (J \cap t^{\lambda} J t^{-\lambda})$.
\end{lemma} 

\begin{proof}
  By definition, $\Phi(f_0) = p_0$. Next, for arbitrary $\lambda$,
  consider the left cosets of $J$ in $J t^{\lambda} J$.  $J$ acts
  transitively on these by left multiplication, so these cosets have
  the form $\{j_i t^{\lambda} J\}$, where $j_i$ is a set of
  representatives of the finite quotient $J / (J \cap t^{\lambda} J
  t^{-\lambda})$.  Thus, $f_\lambda = \sum_i \mu(j_i) j_i
  t^{\lambda} \cdot f_0$, where each $j_i t^{\lambda} \cdot f_0$ is the unique
  function in $\sW$ supported on the left coset $j_i t^\lambda J$
  whose value at $j_i t^\lambda$ is $1$.  Applying $\Phi$ yields the
  result.
\end{proof} 

To prove the desired equality, note that
\begin{equation}\label{e:phifltl}
  \Phi(f_\lambda)(t^\lambda) = \sum_i \mu(j_i) p_0(t^{-\lambda} j_i^{-1} t^\lambda).
\end{equation}
For each $j_i$, write $j_i = j_i^+ j_i^0 j_i^-$ for $j_i^+ \in J^+,
j_i^0 \in J^0$, and $j_i^- \in J^-$.  Substituting \eqref{e:p0cl} and
\eqref{e:tjjm} into \eqref{e:phifltl}, we obtain
\begin{equation} \label{phifltl2} \Phi(f_\lambda)(t^\lambda) = |\{j_i:
  t^{-\lambda} (j_i^-)^{-1} t^{\lambda} \in J^-\}|.
\end{equation}
To conclude, recall that $j_i$ are representatives of $J / (J \cap
t^{\lambda} J t^{-\lambda})$.  Note that the RHS of \eqref{phifltl2}
identifies with $|K / (J \cap t^{\lambda} J t^{-\lambda})| $, where
$K$ has the same form as $J$ except with $f(\alpha)$ replaced by
$f(\alpha) + \max \{\alpha(\lambda),0\}$ for $\alpha \in \Delta_-$
(leaving $f(\alpha)$ the same when $\alpha \in \Delta_+$).  Hence,
$\log_q \Phi(f_\lambda)(t^\lambda) = \sum_{\alpha \in \Delta_+} \max
\{\alpha(\lambda), 0\}$. In view of Lemma \ref{l:clambda1}, this
implies the desired formula. \qed

\subsection{Proof of Proposition \ref{p:irrconn}.(ii)}\label{ss:irrconn}
We will follow to some extent the arguments of \cite[Theorem
4.15]{Roche98}, with an innovation from \cite{Yu-ctsr} to reduce
restrictions on the residue characteristic. Note that \cite{Roche98}
works in the mixed-characteristic setting where $F$ has characteristic
zero (and $\cO/\fp = \Fq$), unlike us. However, as pointed out there,
those arguments extend to our equal-characteristic setting by
replacing Proposition 4.11 there by the more general \cite[Theorem
7.1]{ARirpg}, which is for arbitrary local fields $F$ with residue
field $\Fq$ (see Theorem \ref{t:ar-int} below), proved similarly.

The proof is by induction on the semisimple rank of $G$.  If $G$ is a
torus, then the assumption $x \notin \bJ t^\lambda \bJ$ is vacuous, so
the result follows.  So we assume $G$ has positive semisimple rank and
that the result holds follows for all connected split reductive groups
of strictly smaller rank (for all characters, using Roche's
corresponding subgroup).

  Let $\ell = \cond(\mu) \geq 1$. If $\ell = 1$, then
  $J$ is the Iwahori subgroup, in which case the Bruhat decomposition
  implies that $x \in J n J$ for some $n \in N(T(F))$. In this case,
  we can assume $x = n$, and the result follows from part (i) of the
  proposition. Henceforth, we assume $\ell \geq 2$. 
  
  \subsubsection{Review of some notation used in \cite{Roche98}} 
  Following \cite[\S
  4]{Roche98}, define the groups
\begin{gather*}
  L := \langle T_{\ell-1}, U_{\alpha, \ell-1}, U_{\beta, f(\beta)}
  \mid
  c_\alpha < \ell, c_{\beta} = \ell  \rangle, \\
  K_i := \langle T_{i}, U_{\alpha,i} \mid \alpha \in \Delta \rangle
  \quad (\forall i \geq 1), \\
  \widetilde K_\ell := \begin{cases} K_\ell, & \text{if $\ell$ is even}, \\
    \langle K_\ell, U_{\alpha, \ell-1} \mid \alpha > 0, c_\alpha =
    \ell \rangle, & \text{if $\ell$ is odd}.\end{cases}
\end{gather*}
Note that $L \subseteq J$.  Moreover, $K_{\lfloor \ell/2 \rfloor}
\supseteq L \supseteq K_{\ell-1} \supsetneq \widetilde K_{\ell}$.
Finally, $K_i / K_{2i}$ is abelian for all $i \geq 1$.

Next, we recall the Lie algebras of the above subgroups and
bijections $\mfK_i / \mfK_{2i} \iso K_i/K_{2i}$ from \emph{op.~cit.}
Let $\mfg$, $\mft$, and $\mfu_\alpha$ be the Lie algebras of $G$, $T$,
and $U_\alpha$ over $F$.  Let $X = \Hom(\bT, \bG_m)$ be the lattice of
characters of $\bT$ and $X^\vee = \Hom(\bG_m, \bT)$ the lattice of
coweights. There is a natural map $X^\vee \otimes_{\bZ} F \iso \mft$.
Let $\mft_i \subset \mft$ be the $\cO$-sublattice which is the image
of $\fp^i \otimes_{\bZ} X^\vee$ (note that $\mft_i$ is the Lie algebra
of $T_i$). Similarly, let $\mfu_{\alpha,i} \subset \mfu_{\alpha}$ be
the $\cO$-sublattice which is the image of $\fp^i$ under the
isomorphism $F \iso \mfu_{\alpha}$ defined by the map $\Lie(u_\alpha)$
(note that $\mfu_{\alpha,i}$ is the Lie algebra of
$U_{\alpha,i}$). Define the following $\cO$-sublattices of $\mfg$,
which are the Lie algebras of the groups
$L, K_i$, and $\widetilde K_\ell$:
\begin{gather*}
  \mfL := \mft_{\ell-1} \oplus \bigoplus_{\alpha: c_\alpha < \ell}
  \mfu_{\alpha,\ell-1} \oplus \bigoplus_{\alpha: c_\alpha = \ell}
  \mfu_{\alpha,f(\alpha)}, \\
  \mfK_i := \mft_i \oplus \bigoplus_{\alpha \in \Delta} \mfu_{\alpha,i}, \\
  \widetilde \mfK_\ell := \begin{cases} \mfK_\ell, & \text{if $\ell$ is even}, \\
    \mft_\ell \oplus \bigoplus_{\alpha > 0, c_\alpha = \ell}
    \mfu_{\alpha, \ell-1} \oplus \bigoplus_{\alpha: c_\alpha < \ell
      \text{ or } \alpha < 0} \mfu_{\alpha,\ell}, & \text{if $\ell$ is
      odd}.
\end{cases}
\end{gather*}
Next, for $i \geq 1$, the bijections $\Lie(u_\alpha): \fp^i \iso
\mfu_{\alpha,i}$ and $u_{\alpha}: \fp^i \iso U_{\alpha,i}$ induce a
bijection of sets $\mfu_{\alpha,i} \iso U_{\alpha,i}$.  Similarly, the
bijections $X^\vee \otimes_{\bZ} \fp^i \iso \mft_i$ and $X^\vee
\otimes_{\bZ} (1+\fp^i) \iso T_i$, together with the bijection $\fp^i
\iso (1 + \fp^i), b \mapsto 1+b$, induce a bijection of sets $\mft_i
\iso T_i$.  Using these and the direct product and sum decompositions
$\mfK_i = \mft_i \oplus \bigoplus_{\alpha \in \Delta} \mfu_{\alpha,i}$
and $K_i = T_i \cdot \prod_{\alpha \in \Delta} U_{\alpha,i}$ (for some
choice of ordering of the roots), one obtains a \emph{noncanonical}
bijection
\[
\varphi_i: \mfK_i \iso K_i,
\]
depending on the choice of ordering of $\alpha \in \Delta$.  In fact,
$\varphi_i = \varphi_1 |_{\mfK_i}$ for all $i \geq 1$.  We also have
\[
\varphi_L := \varphi_1|_{\mfL}: \mfL \iso L.
\]
Since $K_i / K_{2i}$ is abelian, the resulting isomorphism
$\overline{\varphi_i}: \mfK_i / \mfK_{2i} \iso K_i / K_{2i}$ is
independent of the ordering of roots and hence
\emph{canonical}. Similarly, the isomorphism $\overline{\varphi_L} :
\mfL / \widetilde \mfK_{\ell} \iso L / \widetilde K_\ell$ is
independent of the ordering of roots and hence canonical.

Let $\psi: F \to \bQlt$ be an additive character such that $\fp
\subseteq \ker(\psi)$ and $\cO \nsubseteq \ker(\psi)$.  Then, $a :=
\mu \circ \overline{\varphi_L}$ is a character of $\mfL / \widetilde
\mfK_{\ell}$, and can be viewed as an element of $(\mft_{\ell-1} /
\mft_{\ell})^*$ (an element acting trivially on the off-diagonal part
$\bigoplus_{\alpha: c_\alpha < \ell} \mfu_{\alpha,\ell-1} \oplus
\bigoplus_{\alpha: c_\alpha = \ell} \mfu_{\alpha,f(\alpha)}$ of
$\mfL$).


Following \cite{Roche98}, define
\[
\mathcal{I}(\mu | H) := \{g \in G(F): \mu(h)=\mu(g^{-1}hg), \forall h \in
H \cap gHg^{-1}\}.
\]
The relationship to our objects of study is: $g \in \mathcal{I}(\mu | H)$
if and only if, for all pairs $(h, g^{-1}hg) \in \Stab_{H \times H}(g)$,
 $\mu(h)\mu(g^{-1}hg)^{-1} = 1$. That is,
\begin{equation} \label{e:imuh}
g \in \mathcal{I}(\mu|H) \iff
\Stab_{H \times H}(g) \subseteq \ker(\mu \times \mu^{-1}).
\end{equation}

\subsubsection{Proof of the proposition in the case $x \in \mathcal{I}(\mu | L)$} 
We will need the following result, which follows from
\cite[Proposition 4.11]{Roche98} and \cite[Theorem 7.1]{ARirpg}
(slightly modifying the proof to use the dual Lie algebra as in
\cite{Yu-ctsr}; for instance, \cite[Lemma 5.1]{Yu-ctsr} replaces
\cite[Lemma 1.8.1]{Araksr} with the same proof).
\begin{theorem}\label{t:ar-int}
$\mathcal{I}(\mu | L) = L
  C_{G(F)}(a) L$.
\end{theorem}

Here, $C_{G(F)}(a)$ is the centralizer in $G(F)$ of $a$. According to
\cite[Proposition 7.3]{Yu-ctsr}, under our restrictions on residue
characteristic, this is the group of $F$-points of a semistandard Levi
subgroup, call it $C_G(a)$, of $G$. As explained in the proof of
\cite[Theorem 4.15]{Roche98}, up to multiplying $\mu$ by a suitable
character of $G$ (which leaves $J$ unchanged, since such characters
are trivial on $[G,G]$ and hence the $c_\alpha$ are unchanged), we can
assume that $C_G(a) \neq G$, and $C_G(a)$ is a connected split
reductive group of strictly lower semisimple rank than $G$.  Then, the
subgroup $J_{\mu,C_{G(F)}(a)} < C_{G(F)}(a)$ associated to $\mu$ is
nothing but the intersection $J_{\mu,C_{G(F)}(a)} = C_{G(F)}(a) \cap
J$.  By induction on the semisimple rank of $G$, we can therefore
assume that, if the element $x$ in the statement of the proposition is
in $C_{G(F)}(a)$, then $\Stab_{\bJ_{\mu,C_{G(F)}(a)} \times
  \bJ_{\mu,C_{G(F)}(a)}}(x)^\circ(\Fq) \nsubseteq \ker(\mu \times
\mu^{-1})$.  Therefore, the proposition follows for $x$.  Hence, it
also follows if $x \in J C_{G(F)}(a) J$, and hence if $x \in L
C_{G(F)}(a) L$.

\subsubsection{Proof of the proposition in the case $x\notin \mathcal{I}(\mu | L)$}\label{sss:irrconn}
By \eqref{e:imuh}, $\Stab_{\bL \times
  \bL}(x)(\Fq) \nsubseteq \ker(\mu \times \mu^{-1})$.  Our goal is to show 
  \[
  \Stab_{\bL \times \bL}(x)^{\circ}(\Fq)
\nsubseteq \ker(\mu \times \mu^{-1})\]

  Note that $\Stab_{\bL \times \bL}(x)(\Fq) = \{(g,x^{-1} g x): g \in
  L \cap xLx^{-1}\} \cong L \cap x Lx^{-1}$.  We will need to recall
  the following observation of \cite{Araksr}.  For all $x \in G(F)$,
  let $\mfK_{x,r} := \mfK_r \cap \Ad(x)\mfK_r$. Similarly define
  $K_{x,r}$ as well as $\widetilde \mfK_{x,\ell}$ and $\widetilde
  K_{x,\ell}$. Then, as observed in \cite[(1.5.2)]{Araksr}, for all $x
  \in G(F)$ and all $r \geq 1$, $\mfK_{x,r} / \mfK_{x,2r}$ is abelian,
  and $\varphi_r$ restricts to an isomorphism
  $\overline{\varphi_{x,r}}: \mfK_{x,r} / \mfK_{x,2r} \iso K_{x,r} /
  K_{x,2r}$, which is independent of the ordering of the roots.

  It follows from the definition of $\varphi_r$ that
  $\overline{\varphi_{x,r}}$ is the map on $\Fq$-points of an
  isomorphism of commutative algebraic groups. In the case that $\ell$
  is even, we deduce by restriction that one also has a canonical
  isomorphism 
\[
(\mfL \cap \Ad(x)\mfL)/\widetilde \mfK_{x,\ell} \iso (L
  \cap \Ad(x)L)/\widetilde K_{x,\ell},
\]
which is the map on $\Fq$-points of an isomorphism of commutative
algebraic groups. It is easy to generalize to the case where $\ell$ is
odd.

  Now, since
  $(\mfL \cap
  \Ad(x)\mfL)/\widetilde \mfK_{x,\ell}$ is the group of
  $\Fq$-points of a product of finitely many
  copies of $\mathbb{G}_a$, the same is true for 
  $(L \cap \Ad(x)L)/\widetilde K_{x,\ell}$.  So, this quotient is
  connected.  Now, since $\{(g,x^{-1} g x): g \in \widetilde
  K_{x,\ell}(\Fq) \} \subseteq \ker(\mu \times \mu^{-1})$, we conclude
  the desired statement.
\qed


\subsection{Completion of the proof of Lemma \ref{l:J'algebraic}}\label{ss:J'algebraic}
  
Let $c_\alpha := f(\alpha) + f(-\alpha)$ (this is consistent with our other
definition $c_\alpha = \cond(\bmu \circ \alpha^\vee)$ when $f = f_{\bmu}$).
For all $m \geq 1$, let $T_{f,m}$ be the subtorus of $T$ generated by
all coroots $\alpha^\vee$ such that $c_\alpha \leq m$. Clearly,
$T_{f,i} \leq T_{f,j}$ for $i \leq j$.  Furthermore, $T \cap [G,G] =
T_{f,m}$ for $m \geq \max_{\alpha \in \Delta} c_\alpha$.  Observe that
$T_f$ is the product (not direct) of $T_{f,m}(1+\fp^m)$ for all $m$,
where for any algebraic subtorus $S < T$, $S(1 + \fp^m)$ denotes the
subgroup of $S(\cO)$ generated by the coweights of $S$ evaluated at
$1+\fp^m$.  It follows easily that one has an isomorphism of groups 
\begin{equation}
  T_{f} \cong T_{f,1}(1+\fp) 
\times \prod_{m \geq 2} (T_{f,m}/T_{f,m-1})(1+\fp^m).
\end{equation}
From this one sees that $T_f$ is the $\Fq$-points of a canonical
proalgebraic (and prounipotent) subgroup of $\bT_\cO$.  Similarly, for
$m \geq \max_{\alpha \in \Delta} c_\alpha$,
\begin{equation}
\langle T_f, T_m \rangle \cong T_f \times
(T/T_{f,m})(1+\fp^m),
\end{equation}
and one concludes also that this is the $\Fq$-points of a canonical
proalgebraic subgroup of $\bT_{\cO}$ (depending on $c$ as well as
$f$).  Finally, since $A = T(\cO)/T'$ is a quotient of $T(\cO) / T_m$,
which is finite, so is $A$, and the geometric version of this
statement is that $\bA$ is an algebraic group (of finite type).

Applying the above to $f = f_{\bmu}$, one sees that $T_{f_{\bmu}}$ and
$T'$ are the $\Fq$-points of canonical proalgebraic subgroups of
$\bT_{\cO}$, and that $\bA \cong \bT_{\cO}/\bT'$ is an algebraic group.

\subsection{Proof of Proposition \ref{p:conv}} \label{ss:conv}
 
Let $\bJ^{\lambda,\nu}:=\bJ
t^{\lambda}\bJ\times_{\bJ}\bJ^{\nu}\subseteq \bG\times_{\bJ} \bX$ and
let $p^{\lambda,\nu}:\bJ^{\lambda,\nu}\ra \bX$ denote the restriction
of $p$ to $\bJ^{\lambda,\nu}$.  Note that there is a natural action of
$\bJ \times \bA$ on $\bJ^{\lambda,\nu}$ given by left multiplication
by $\bJ$ (on the first factor, $\bJ t^\lambda \bJ$) and right multiplication
by $\bA$ (on the second factor, $\bJ^\nu$), where, by convention, we
use the right multiplication action $a \cdot_R j := j a^{-1}$ (even
though $\bA$ is commutative).
  
\begin{lemma} \label{l:convDef} For all $\lambda,\nu \in \Lambda$,
\begin{enumerate} 
\item[(i)] $\cF_{\lambda}'\tboxtimes \cF_{\nu}'$ is a $(\bJ \times
  \bA, \cM \boxtimes \cM_0^{-1})$-equivariant local system on
  $\bJ^{\lambda,\nu}$.
\item[(ii)] $j_{!}^{\lambda} \star j_!^{\nu} \cong p_{!}^{\lambda,\nu}
  (\cF_{\lambda}\tboxtimes \cF_{\nu})$.
\item[(iii)] $j_{!}^{\lambda} \star j_!^{\nu} \in \sDg  (\bX)$.
\item[(iv)] $j_{!}^{\lambda} \star j_!^{\nu}$ is $(\bJ\times \bA,
  \cM\boxtimes \cM_0^{-1})$-equivariant.
\end{enumerate} 
\end{lemma}

\begin{proof} 
  By Lemma \ref{l:prodLocalSys}, $\cF_{\lambda}' \tboxtimes
  \cF_{\nu}'$ is a local system on $\bJ^{\lambda,\nu}$. Since
  $\cF_{\lambda}'$ and $\cF_{\nu}'$ are both equivariant, so is
  $\cF_{\lambda}' \tboxtimes \cF_{\nu}'$; thus, (i) is
  established. For (ii) see Lemma \ref{l:prodLocalSys2}. Since
  $\bJ^{\lambda,\nu}$ (and therefore $p^{\lambda,\nu}$) is affine,
  Theorem \ref{t:Artin} implies (iii). Finally, for (iv) we apply the
  projection formula to compute
\begin{multline*}
  l^* (j_!^\lambda \star j_!^\nu) = l^* p_!^{\lambda,\nu} (\cF_\lambda
  \tboxtimes \cF_\nu) \cong (\Id_{\bJ} \times p^{\lambda,\nu})_! l^*
  (\cF_\lambda \tboxtimes \cF_\nu) \cong \\ (\Id_{\bJ} \times
  p^{\lambda,\nu})_!  (\cM\boxtimes (\cF_\lambda \tboxtimes \cF_\nu))
  \cong \cM \boxtimes p_!^{\lambda,\nu}(\cF_\lambda \tboxtimes
  \cF_\nu) \cong \cM \boxtimes (j_!^\lambda \star j_!^\nu). \qedhere
\end{multline*}
\end{proof}

Next, we prove that $j_!^\lambda \star j_!^\nu\cong j_!^{\lambda+\nu}$
in three steps:

\underline{Step I: $\lambda$ and $\nu$ are both dominant (or both
  antidominant)}: In this case, Corollary \ref{c:proddom} implies that
$\bJ t^\lambda \bJ \times_{\bJ} \bJ^\nu \cong \bJ^{\lambda+\nu}$ by
the multiplication map, and making this identification,
$p^{\lambda,\nu}$ becomes the inclusion $j^{\lambda+\nu}$.  It is
clear that $\cF_\lambda \tboxtimes \cF_\nu$ is a $(\bJ \times \bA, \cM
\times \cM_0^{-1})$-equivariant rank-one local system. It follows from
the definitions that $(\cF_\lambda \tboxtimes
\cF_\nu)|_{t^{\lambda+\nu} \bJ'} \cong
\cF_{\lambda+\nu}|_{t^{\lambda+\nu} \bJ'}$. We deduce from Lemma
\ref{l:Flambda}.(iii) that
$\cF_\lambda \tboxtimes \cF_\nu \cong \cF_{\lambda+\nu}$. \\

  \underline{Step II: $\lambda$ is dominant, $\nu$ is antidominant}:
  Let $x\in \bX$. We claim that $(j_!^\lambda \star j_!^\nu)_{x}=0$ if
  $x\notin \bJ^{\lambda+\nu}$.  Indeed, by Lemma \ref{l:convDef}.(ii),
  $j_!^\lambda \star j_!^\nu \cong p_{!}^{\lambda,\nu}
  (\cF_{\lambda}\tboxtimes \cF_{\nu})$.  By Proposition
  \ref{p:restrictionIrrelevant} the stalk of this complex at $x\in
  \bX$ is nonzero only if $x$ is a relevant point in the image of
  $p^{\lambda,\nu}$. By Proposition \ref{p:cosetMult}.(c), the only
  relevant orbit inside this image is $\bJ^{\lambda+\nu}$.

  Now we claim that we are in a position to apply Theorem
  \ref{t:semismall} to prove that $j_!^\lambda \star j_!^\nu$ is
  perverse.  It is clear that $p^{\lambda,\nu}$ is an affine
  morphism. Next, let $\cP$ denote the partition of the closure of the
  image of $p^{\lambda,\nu}$ consisting of three locally closed
  subschemes: $\bJ^{\lambda+\nu}, \overline{\bJ^{\lambda+\nu}}
  \setminus \bJ^{\lambda+\nu}$, and the complement of
  $\overline{\bJ^{\lambda+\nu}}$. Note that, as the closure of the
  image of $p^{\lambda,\nu}$ is irreducible, one of these locally
  closed subschemes must, in fact, be open and dense. By Proposition
  \ref{p:semismall}, for every closed point $x\in \bJ^{\lambda+\nu}$,
\[
\dim((p^{\lambda,\nu})^{-1}(x)) = \frac{1}{2}
[\dim(\bJ^{\lambda})+\dim(\bJ^{\nu}) - \dim(\bJ^{\lambda+\nu})].
\]
From this, it follows that $p^{\lambda,\nu}$ is semismall at every
$x\in \bJ^{\lambda+\nu}$ (at non-closed points $y$, the LHS should be
replaced by the dimension of the generic fiber at closed points in the
closure of $y$, cf.~\eqref{e:semismall}, and the result follows from
the one for closed points). Since the stalk of $j_!^\lambda \star
j_!^\nu$ at every point outside $\bJ^{\lambda+\nu}$ vanishes, Theorem
\ref{t:semismall} shows that $j_!^\lambda \star j_!^\nu \cong
p^{\lambda,\nu}_!  (\cF_{\lambda}\tboxtimes \cF_{\nu})$ is perverse.

Since $j_!^\lambda\star j_!^\nu$ is perverse and its stalks vanish
outside of $\bJ^{\lambda+\nu}$, it must be isomorphic to $j_! j^*
(j_!^\lambda\star j_!^\nu)$, where $j=j^{\lambda+\nu}$.  Let $\cF$ be the restriction of $j_!^\lambda \star
j_!^\nu$ to $\bJ^{\lambda+\nu}$. This is a perverse sheaf, hence it must be a local system on an open subvariety. Since it is $(\bJ\times \bA, \cM\boxtimes \cM_{0}^{-1})$-equivariant, we conclude that it is, in fact, a local system. Lemma 
\ref{l:Flambda}.(iii) implies that $\cF$ is isomorphic to $\cF_{\lambda+\nu}$.\\

\underline{Step III: $\lambda$ and $\nu$ arbitrary}: 
 Write $\lambda = \lambda_+ - \lambda_-$ and $\nu =
  \nu_+ - \nu_-$ for $\lambda_+, \lambda_-, \nu_+, \nu_- \in
  \Lambda_+$.  Moreover, we can arrange this so that $\lambda_- =
  \nu_+$.  Then 
  \[
  j_{!}^\lambda \star j_{!}^\nu = (j_!^{\lambda_+} \star
  j_!^{-\lambda_-}) \star (j_!^{\nu_+} \star j_!^{-\nu_-}) =
  j_!^{\lambda_+} \star (j_!^{-\lambda_-} \star j_!^{\lambda_-}) \star
  j_!^{-\nu_-} = j_!^{\lambda_+} \star j_!^0 \star j_!^{-\nu_-} =
  j_{!}^{\lambda_{+}-\nu_{-}}=j_{!}^{\lambda+\nu}.  \qed
\]

\subsection{Proof of Corollary \ref{c:BMWExt}} \label{ss:BMWExt} Note
that $J^0$ is closed.  Let $\bone:=j_!^0 \cong j_*^0$. Then for all
$\cF\star \bone \cong \bone \star \cF\cong \cF$ for all $\cF\in \sHg$.
Let $\cF,\cG, \cH\in \sHg$ and assume that $\cH\star \cF$ and
$\cH\star \cG$ are in $\sHg$. Then $\cH\star -$ defines a homomorphism
$\Hom(\cF,\cG)\ra \Hom(\cH\star \cF, \cH\star \cG)$. Now assume there
exists $\cH'\in \sHg$ such that $\cH'\star \cH=\bone$. Then the
composition
\[
\Hom(\cF,\cG)\ra \Hom(\cH\star \cF, \cH\star \cG) \ra \Hom(\cH'\star
\cH\star \cF, \cH'\star \cH\star \cG) = \Hom(\bone \star \cF,\bone
\star \cG)=\Hom(\cF,\cG)
\]
is the identity. Similarly, the composition 
\[
\Hom(\cH\star \cF, \cH\star \cG) \to \Hom(\cH'\star \cH\star \cF,
\cH'\star \cH\star \cG) =\Hom(\cF,\cG) \to \Hom(\cH\star \cF, \cH\star
\cG)
\]
is the identity.  Hence $\Hom(\cF, \cG)\cong \Hom(\cH\star \cF,
\cH\star \cG)$. The same holds when we replace $\Hom$ by
$\Ext^\bullet$ or take convolution on the right instead of the left
(with $\cH$ having a right inverse).

Applying the above considerations and Proposition \ref{p:conv} we
conclude
\begin{multline*}
  \Ext^\bullet(j_!^\lambda \otimes \cL, j_!^\nu \otimes \cK) =
  \Ext^\bullet(j_!^\lambda \star j_!^{-\nu} \otimes \cL, j_!^\nu \star
  j_!^{-\nu} \otimes \cK) = \Ext^\bullet(j_!^{\lambda-\nu} \otimes
  \cL, j_!^0 \otimes \cK)= \\ = \Ext^\bullet(j_!^{\lambda-\nu} \otimes
  \cL, j_*^0 \otimes \cK) = \Ext_{\Spec \Fq}^\bullet((j^0)^*
  j_!^{\lambda-\nu} \otimes \cL, \cF_0 \otimes \cK).
\end{multline*}
This is zero unless $\lambda= \nu$, in which case it is $\Ext_{\Spec
  \Fq}^\bullet(\cF_0 \otimes \cL, \cF_0 \otimes \cK) = \Ext_{\Spec
  \Fq}^\bullet(\cL, \cK)$.\qed

\subsection{Proof of Proposition \ref{p:adjoint}} \label{ss:adjoint}  
Let $p_\lambda$ denote the multiplication morphism $\bG\times_\bJ
\bJ^\lambda \to \bX$. 
Let
$d=\dim(\bJ^\lambda)$. Then,
\begin{align*}
 \Hom(\cF\star_! j_!^\lambda, \cG) 		
					& =  \Hom((p_\lambda)_!(\cF
\tboxtimes \cF_\lambda), \cG)\\
& \cong	 \Hom(\cF \tboxtimes \cF_\lambda, p_\lambda^! \cG) \\
& \cong	 \Hom(\cF \tboxtimes \cF_\lambda, p_\lambda^* \cG[2d](d)).
\end{align*}
In the last isomorphism, we used that
$p_\lambda^!=p_\lambda^*[2d](d)$, since $p_\lambda$ is a smooth
morphism of relative dimension $d$.

\begin{claim}\label{c:adj}
$\Hom(\cF \tboxtimes \cF_\lambda, p_\lambda^* \cG[2d](d))
\cong \Hom(p_{-\lambda}^* \cF, \cG \tboxtimes \cF_{-\lambda})$.
\end{claim}

Using the claim, we can easily complete the proof similarly to the above:
\begin{align*}
\Hom(p_{-\lambda}^* \cF, \cG \tboxtimes \cF_{-\lambda}) & \cong 
\Hom(\cF, (p_{-\lambda})_*(\cG \tboxtimes \cF_{-\lambda})) \\
& \cong  \Hom(\cF, \cG \star_* j^{-\lambda}_*).
\end{align*}

It remains to prove Claim \ref{c:adj}. The proof relies on converting
between the functors $- \tboxtimes \cF_\lambda$ and $p_\lambda^* -$.
We first explain how to do this in a simpler (and probably standard)
situation, where $\bG$ is an algebraic group (i.e., of finite type),
$\bH \subseteq \bG$ is a subvariety, we replace $\cF_\lambda$ by
$\bQl|_{\bH}$, and eliminate the twists and twisted products. The
analogous claim in this simpler situation can be formulated as
follows. Let $\tilde p:\bG\times \bG\ra \bG$ denote the multiplication
map (we use tildes to distinguish from the maps we will define
eventually on the level of $\bG \times_\bJ \bX$ and in the twisted
setting.) Let $\cF, \cG \in \sD(\bG)$. Let $\bH^{-1}$ be the image of
$\bH$ under the inversion automorphism of $\bG$.
   
  \begin{claim} \label{c:adjSimple} 
  $\Hom_{\bG\times \bH} (\cF\boxtimes \bQl, \tilde p^* \cG|_{\bG\times \bH}) \cong 
  \Hom_{\bG\times \bH^{-1}} (\tilde p^* \cF|_{\bG\times \bH^{-1}} ,\cG\boxtimes \bQl)$.
  \end{claim} 
  
  \begin{proof} Let $\tilde
  \Gamma: \bG \times \bG \to \bG \times \bG$ denote the isomorphism $\tilde
  \Gamma(g,x)=(gx,x)$. Then there is a commutative diagram 
\begin{equation}\label{e:gamma-square}
  \xymatrix{
    \bG \times \bG \ar[r]^{\tilde \Gamma} \ar[d]^{\tilde p} & \bG \times \bG 
    \ar[d]^{\tilde \pi_1} \\
    \bG \ar@{=}[r] & \bG.
  }
\end{equation}
Therefore, $\tilde p = \tilde \pi_1 \tilde \Gamma$, and
hence $\tilde \Gamma^*(\cF \boxtimes \bQl)\cong \tilde \Gamma^*\pi_1^* \cF  \cong \tilde p^* \cF$.

Next, let 
$\tilde \iota_2: \bG \times \bG \to \bG \times \bG$ denote the isomorphism $\tilde
\iota_2(g,x) = (g,x^{-1})$.  We need the identity
\begin{equation}\label{e:gi-id}
  (\tilde \Gamma \tilde \iota_2)^2 = \Id_{\bG\times \bG} = (\tilde \iota_2 \tilde \Gamma)^2.
\end{equation}
Finally,
\begin{align*}
  \Hom_{\bG \times \bH}(\cF \boxtimes \bQl, \tilde p^* \cG|_{\bG
    \times \bH}) & \cong
  \Hom_{\bG \times \bH^{-1}}(\cF \boxtimes \bQl, \tilde \iota_2^* \tilde p^* \cG|_{\bG \times \bH^{-1}}) \\
  & \cong  \Hom_{\bG \times \bH^{-1}} (\tilde p^* \cF|_{\bG \times \bH^{-1}}, \tilde \Gamma^* \tilde \iota_2^* \tilde p^* \cG|_{\bG \times \bH^{-1}}) \\
  & \cong \Hom_{\bG \times \bH^{-1}}(\tilde p^* \cF|_{\bG \times
    \bH^{-1}}, ((\tilde \iota_2)^{-1})^* (\tilde \Gamma^{-1})^* \tilde
  p^* \cG|_{\bG \times \bH^{-1}}) \\ & \cong \Hom_{\bG \times
    \bH^{-1}}(\tilde p^* \cF|_{\bG \times \bH^{-1}}, \cG \boxtimes
  \bQl),
\end{align*}
as desired.
\end{proof}

\begin{proof}[Proof of Claim \ref{c:adj}] Let $\bK^\lambda=\bJ t^\lambda \bJ \subset \bG$. We want to imitate the proof of the previous claim with $\bH$ replaced by $\bK^\lambda$. 
Since we have
difficulty working with objects over $\bG\times \bK^\lambda$, as it is not ind-finite,
we instead descend $\tilde \Gamma|_{\bG \times \bK^\lambda}$ to an
isomorphism
\begin{equation}
  \Gamma_\lambda: \bG \times_\bJ \bJ^\lambda \iso (\bG \times \bJ\backslash\bG)/\bJ', \quad 
  (g,x) \mapsto (gx,x),
\end{equation}
where $\bJ'$ acts diagonally by $(g,x) \cdot j := (gj, xj)$. 

Similarly, the inversion map in the second component, $\tilde
\iota_2$, descends to
\begin{equation}
\iota_2: \bG \times_\bJ \bX \iso (\bG \times \bJ'\backslash\bG)/\bJ, \quad 
(g,x) \mapsto (g,x^{-1}),
\end{equation}
where again $\bJ$ acts diagonally by $(g,x) \cdot j := (gj, xj)$.

We will need the equivalence
\begin{equation}\label{e:perv-tw-equiv}
  \tau_\lambda: \Perv_{(\bA,\cM_0^{-1})}((\bG \times \bJ'\backslash\bK^\lambda)/\bJ) \iso
  \Perv_{(\bA,\cM_0)}((\bG \times \bJ\backslash\bK^\lambda)/\bJ'), \quad
  \cG \mapsto \cG \otimes \overline{(\bQl \boxtimes (\cF'_\lambda)^{-1})}.
\end{equation}
Here $\overline{(\bQl \boxtimes (\cF'_\lambda)^{-1})}$ is the local system $\bJ'\backslash \bK^\lambda)/\bJ'$
obtained by equivariant descent from the local system $\bQl \boxtimes
(\cF'_\lambda)^{-1}$ on $\bG \times \bK^\lambda$, and we view
both categories in \eqref{e:perv-tw-equiv} as categories of
twisted-equivariant perverse sheaves on $(\bG \times
\bJ'\backslash \bK^\lambda)/\bJ'$ (lifting from a quotient by
$\bJ$ to ordinary $\bA$-equivariant objects on the quotient by
$\bJ'$).  

The identity analogous to  \eqref{e:gi-id} in this twisted setting is
\begin{equation}\label{e:cr-inv}
(\Gamma_\lambda^* \tau_\lambda (\iota_2^{-1})^*)^{-1} \cong
\Gamma_{-\lambda}^* \tau_{-\lambda} (\iota_2^{-1})^*.
\end{equation}

From now on, an overlined quantity means the object living over the
appropriate base (indicated by the subscript of $\Hom$) obtained by
equivariant descent. For instance, $\cF \tboxtimes \cG = \overline{\cF
  \boxtimes \cG}$.  Also, note that $\cF_\lambda' = \overline{\cM
  \times \cM^{-1}}$ for all $\lambda$, working over the base
$\bJ^\lambda$ (which is a quotient of $\bJ \times \bJ$).  Then,
\begin{align*}
\Hom_{\bG \times_\bJ \bJ^\lambda}(\cF \tboxtimes \cF_\lambda', p_\lambda^* \cG) & \cong 
\Hom_{(\bG \times \bJ'\backslash \bK^{-\lambda})/\bJ}((\iota_2^{-1})^*(\cF \tboxtimes \cF_\lambda'), (\iota_2^{-1})^* p_\lambda^* \cG) \\
& \cong 
\Hom_{(\bG \times \bJ'\backslash\bK^{-\lambda})/\bJ}(\overline{\cF \boxtimes (\cM \times \cM^{-1})}, (\iota_2^{-1})^* \Gamma_\lambda^* (\cG \tboxtimes \bQl)) \\
& \mathop{\cong}^{\tau_{-\lambda}}  \Hom_{(\bG \times \bJ\backslash \bK^{-\lambda} )/\bJ'}(\overline{\cF \boxtimes \bQl}, \tau_{-\lambda} (\iota_2^{-1})^* \Gamma_\lambda^*(\cG \tboxtimes \bQl) )
\\ & \mathop{\cong}^{\Gamma_{-\lambda}^*} 
\Hom_{\bG \times_\bJ \bJ^{-\lambda}}(p_{-\lambda}^* \cF, 
\Gamma_{-\lambda}^*\tau_{-\lambda} (\iota_2^{-1})^* \Gamma_\lambda^*(\cG \tboxtimes \bQl) )
\\ & \mathop{\cong}^{\eqref{e:cr-inv}} 
\Hom_{\bG \times_\bJ \bJ^{-\lambda}}(p_{-\lambda}^* \cF, 
\iota_2^* \tau_\lambda^{-1}  (\cG \tboxtimes \bQl)) 
\\ & \mathop{\cong} 
\Hom_{\bG \times_\bJ \bJ^{-\lambda}}(p_{-\lambda}^* \cF, 
\cG \tboxtimes \cF'_{-\lambda}).
\end{align*}
Now, the same computation with the appropriate shifts and Tate twists 
(using Remark \ref{r:clambdaVol}) yields the desired
result.
\end{proof}

\section{Recollections on perverse sheaves} \label{s:perverse}

\subsection{Definition of perverse sheaves} \label{ss:defPerv} Let $X$
be a connected scheme of finite type over a field $k$, which we assume to
be finite or algebraically closed.  Fix a prime
$\ell$ invertible in $k$. Let $\sD(X)$ denote the derived category of
$\bQl$-sheaves on $X$ with bounded constructible cohomology \cite[\S
1.1.2-1.1.3]{WeilII}. Let $\sDl(X)\subseteq \sD(X)$ denote the full
subcategory consisting of all complexes $\cK$ such that
\begin{equation} \label{eq:perverse}
\dim \supp \,\H^{i}(\cK) \leq -i, \quad \quad \forall \, i\in \bZ. 
\end{equation}
Equivalently, 
for all (not necessarily closed) points $x \in X$,
  \begin{equation}\label{e:stalks}
    \mathrm{H}^{n}(\cK_{x}) = 0 \quad \quad 
    \forall \, \, n > - \dim(x).
\end{equation} 
Using Verdier duality (or the notion of cosupport), one similarly
defines $\sDg(X)$ (see, e.g., \cite{BBD82}). The category of perverse
sheaves is defined by
\[
\Perv(X):= \sDg(X)\cap \sDl(X).
\]

The following theorem is essentially due to Artin; see \cite[Theorem
4.1.1]{BBD82}.

\begin{theorem} \label{t:Artin} If $f:X\ra Y$ is an affine morphism of
  separated schemes of finite type over $k$, the functor $f_{*}:
  \sD(X)\ra \sD(Y)$ takes $\sDl(X)$ into $\sDl(Y)$. By Verdier
  duality, this is equivalent to saying that the functor $f_{!}$ takes
  $\sDg(X)$ into $\sDg(Y)$.
\end{theorem} 

\subsection{Intermediate extensions and cleanness}\label{ss:clean-app}
Let $j: Y \into X$ be an embedding of a locally closed subvariety $Y$.
Recall the intermediate extension function $j_{!*}: \Perv(Y) \to
\Perv(X)$, which has the properties $j^* j_{!*} \cF \cong \cF$, and
which takes irreducible perverse sheaves to irreducible perverse
sheaves.
\begin{definition}\label{d:clean}
  Let $\cF \in \Perv(Y)$.  The intermediate extension $j_{!*} \cF$ is
  called \emph{clean} (or a ``clean extension'') if $j_{!*}
  \mathcal{F} \cong j_! \mathcal{F}$.
\end{definition}

\subsection{Semismall morphisms}\label{ss:semismall}

The standard reference for semismall morphisms and their relationship
to perverse sheaves is \cite[\S 6.2]{GM}. Here we follow the treatment
of \cite[\S III.7]{Kiehl00}, since this reference does not assume
properness.

A partition $\cP$ of $Y$ is a collection $\{Y_\alpha \}$ of disjoint
locally closed subschemes of $Y$ such that
\begin{enumerate}
\item [(i)] $Y=\bigsqcup Y_\alpha$,
\item[(ii)] one of these subschemes is open and dense. 
\end{enumerate} 

\begin{definition}\label{d:semismall}
  Let $Y$ be a separated scheme of finite type over $k$. Let $\cP$ be
  a partition of $Y$. Let $y \in Y_\alpha \subseteq Y$ be a possibly
  non-closed point. A morphism of separated schemes $f:X\ra Y$ is
  called \emph{semismall at $y\in Y_\alpha\subseteq Y$ with respect to
    $\cP$} if
\begin{equation}\label{e:semismall}
  \dim(f^{-1}(y)) - \dim(y) \leq \frac{1}{2}[\dim(X)-\dim(Y_{\alpha})].
\end{equation}
\end{definition}

 For the following result, see \cite[Lemma 7.4]{Kiehl00}.

 \begin{lemma} \label{l:semismall} Let $\cF$ be a constructible
   $\bQl$-sheaf on $X$. Suppose $f:X\ra Y$ is a morphism of separated
   schemes of finite type over $k$. Let $\cP$ be a finite partition of
   $Y$. If $f$ is semismall at every point $y\in Y$ with respect to
   $\cP$, then $f_{!}(\cF[\dim(X)]) \in \sDl(Y)$.
\end{lemma}

\begin{theorem} \label{t:semismall} Let $f:X\ra Y$ be an affine
  morphism of separated schemes of finite type over $k$. Let $\cL$ be
  a local system on $X$ and set $\cK:=f_{!}(\cL[\dim X])$. Let $\cP$
  be a finite partition of $Y$. Assume that for every $y\in Y$ either
\begin{enumerate} 
\item[(i)] $\cK_{y}=0$, or
\item[(ii)] $f$ is semismall at $y$ with respect to $\cP$.
\end{enumerate} 
Then $\cK\in \Perv(Y)$. 
\end{theorem} 

\begin{proof} By Theorem \ref{t:Artin}, $\cK\in \sDg(X)$. It remains
  to show that $\cK\in \sDl(X)$. According to \eqref{e:stalks},
  it is sufficient to check the required vanishing at each stalk $y\in
  Y$. If $y$ is not in the image of $f$, then $\cK_{y}=0$ and the
  condition is automatically satisfied. This is also the case if $y$
  is in the image of $f$ and $\cK_{y}=0$. If we are in neither
  situation, then the result follows from Lemma \ref{l:semismall}.
\end{proof}

\subsection{Twisted equivariant sheaves}  Let $G$ be
a connected algebraic group over $k$. Recall that this means that $G$
is a smooth connected group scheme of finite type over $k$. Let
$m:G\times G\ra G$ denote the multiplication.

\begin{definition}\label{d:multLocalSystem} 
  A \emph{one-dimensional character sheaf} on $G$ is a local system
  $\cL$ satisfying $m^{*} \cL \cong \cL\boxtimes
  \cL$.
\end{definition}

\begin{remark} 
Another name for a one-dimensional character sheaf is a
\emph{multiplicative local system}. We note that character sheaves are usually defined to be
    irreducible perverse sheaves on a group over an
    algebraically closed field. It is, therefore, more appropriate to call $\cL\otimes_k {\bar{k}}[\dim(G)]$ a one-dimensional character sheaf. Working over an arbitrary field and ignoring the shift is, however, 
    more convenient for our purposes.
 \end{remark}

Let $\cL$ be a one-dimensional character sheaf on $G$. Let $X$ be a separated scheme of finite type over $k$ equipped with
  an action $a: G \times X \to X$.
    
  \begin{definition}\label{d:equivPerv}
  The category $\sP_{(G,\cL)}(X)$ of
  $(G,\cL)$-equivariant perverse sheaves on $X$ is the full
  subcategory of $\Perv(X)$ consisting of perverse sheaves $\cF$
  satisfying $a^*\cF\cong \cL\boxtimes \cF$. 
\end{definition}

If $\cL$ is
    trivial, we recover the usual notion of equivariant
    perverse sheaves; see \cite[\S 0]{Lusztig84}.


\begin{remark} \label{r:AEquiv} Let $G'$ be a connected subgroup of $G$ and let $A:=G/G'$. Let $\cL_0$ be a
  one-dimensional character sheaf on $A$, and let $\cL$ be its
  pullback to $G$. Suppose $G$ (and therefore $G'$) acts freely on
  $X$. Let $X'=G'\backslash X$. Note that $A$ acts freely on $X'$ and $X=A\backslash X'$.  Let $r:X\ra X'$
  denote the canonical projection. Then
\[
r^*[\dim(G)]: \sP_{(A,\cL_0)} (X') \ra \sP_{(G,\cL)}(X)
\]
 is an equivalence of categories.
 \end{remark} 
 
 \subsubsection{Support of twisted equivariant sheaves} 
Given an algebra $R$ over $k$, an $R$-point $x$ of $X$ is a morphism
$x: \Spec R \to X$. Now the stabilizer $G_x$ is the sub-group scheme
of $G$ fixing the map $x$.  If $x \in X$ is a set-theoretic point, we
can think of it as a point in the above sense in the standard manner,
by letting $R$ be the algebra of functions on $x$ (an extension field
of $k$).

Let $G$ be an algebraic group and $\mathcal{L}$ be a nontrivial
one-dimensional character sheaf on $G$. Suppose $G$ acts on a variety
$X$. Let $x$ be a set-theoretic point of $X$ and let $G_x$ denote the
stabilizer of $X$. Then $G_x$ is a subgroup of $G$. Let $G_x^\circ$
denote the connected component of the identity of $G_x$.  Let
$\mathcal{F}$ be a $(G,\mathcal{L})$-equivariant sheaf on $X$.

\begin{lemma} \label{l:restrictionTrivial}  If the restriction
  $\mathcal{\cL}|_{G_x^{\circ}}$ is nontrivial, then
  the restriction $\mathcal{F}|_x:=x^*\cF$ is zero.
\end{lemma}

\begin{proof}
  The restriction $x^* \mathcal{F}$ is an $(G_x^\circ,
  \mathcal{L}|_{G_x^{\circ}})$-equivariant local system on $x$,
  with respect to the trivial action.  Let $\pi: \{x\} \times
  G_x^\circ \onto \{x\}$ denote the projection, which is also the
  action map.  Since $x^* \mathcal{F}$ is equivariant, $x^*
  \mathcal{F} \boxtimes \bQl = \pi^* x^* \mathcal{F} \cong x^*
  \mathcal{F} \boxtimes \mathcal{L}|_{G_x^\circ}$.  However, by
  assumption, the first local system is constant in the $G_x^\circ$
  direction, but if $x^* \mathcal{F}$ is nonzero, the second is not.
  Hence, $x^* \mathcal{F}$ is zero.
\end{proof}

\subsection{Alternative definitions of twisted sheaves}\label{ss:alternativeTwist}
 The purpose of this subsection is to expand on Remark \ref{r:twperv}. The discussions of this subsection are not used anywhere else in the paper. 

\subsubsection{Central extensions and one-dimensional character sheaves} 
Let $\cL$ be a one-dimensional character sheaf on $G$. 
 Let $\pi_1(G)=\pi_1(G,e)$ denote the algebraic fundamental group of
  $G$. It is well known that the local system $\cL$ defines a
  homomorphism $\pi_1(G)\ra \bQlt$. Let us assume that this
  homomorphism factors through $\chi_\cL:B_\cL \ra \bQlt$, where $B_\cL$ is a
  finite quotient of $\pi_1(G)$. The local systems we consider in this article satisfy this property. In this situation, the epimorphism $\pi_1(G)\onto B_\cL$
  defines a finite covering $\tG \ra G$. Using the fact that $\cL$ is
  multiplicative, one can show that $\tG$ is a central extension of
  $G$; see the introduction of \cite{masoud09}. Thus, we obtain a central extension 
  \begin{equation} \label{eq:centralChar} 
  1\ra B_\cL\ra \tG\ra G\ra 1
  \end{equation} 
  in the category of algebraic groups over $k$. 
   \begin{remark}
     If $k$ has positive characteristic, there exist \'{e}tale covers
     of $G$ which cannot be endowed with the structure of a central
     extension of $G$; see \cite[\S 2.4 and \S B.4]{masoud09}.
    \end{remark} 
 
\subsubsection{Twisted sheaves via gerbes} 
Let $Y$ denote the quotient
  stack $G\backslash X$. Note that $Y$ is an Artin stack.\footnote{For $\ell$-adic sheaves on an Artin stacks see
  \cite{PerverseArtin}.} The central extension (\ref{eq:centralChar}) gives rise to a homomorphism $H^1(Y,G)\ra H^2(Y,B_\cL)$; see \cite{Giraud71}. Composing with the morphism $H^2(Y,B_\cL)\ra H^2(Y,\bQlt)$ defined by $\chi_\cL:B_\cL\ra \bQlt$, we obtain a morphism 
  \begin{equation}\label{eq:torsGerb}
  H^1(Y,G)\ra H^2(Y,\bQlt).
  \end{equation}
 The scheme $X$ is a $G$-torsor on $Y$; therefore, it defines an element in $H^1(Y,G)$. Let $\sL$ denote the
  $\bQlt$-gerbe on $Y$ defined by the image of this element under the morphism (\ref{eq:torsGerb}). 
Then the notion of $(G,\cL)$-twisted equivariant (perverse) sheaf
coincides with the notion of $\sL$-twisted perverse sheaf on $Y$. We note that the idea of twisting sheaves
by gerbes goes back to \cite{Giraud71}. In \cite{Reich10}, Reich applies
twisting to the constructible derived category and the
category of perverse sheaves. 

\subsubsection{Twisted sheaves via equivariant sheaves} 
The character sheaf $\cL$ pulls back to a trivial local system on
  $\tG$. 
   Therefore, $(G,\cL)$-equivariant perverse sheaves on $X$ are
  automatically $\tG$-equivariant, where $\tG$ acts on $X$ via the
  natural map $\tG\onto G$. Moreover, one can show that we have an
  equivalence of categories between $(G,\cL)$-equivariant perverse
  sheaves on $X$ and the full abelian subcategory of perverse sheaves
  on the algebraic stack $\tG\backslash X$ whose pullback to $X$ are
  $(B_\cL,\chi_\cL)$-equivariant (i.e., $B_\cL$ acts on the fibers by $\chi_\cL$). 

\subsection{Twisted external product
  $\tboxtimes$} \label{ss:twistedProd} The notion of twisted external product
of perverse sheaves has been used widely (e.g., in \cite[\S
1.4]{Frenkel01}, \cite[\S 0.2]{Gaitsgory01}, \cite[\S 4]{MV} and
\cite[\S 2.2]{Nadler05}).  In this subsection, we given an overview of this construction and apply it to twisted equivariant sheaves.  

Let $Y$ and $Z$ be separated schemes of
finite type over a field $k$. Let $H$ be a connected algebraic group
over $k$ and let $p:X\ra Y$ be a right $H$-torsor. Suppose $H$ acts on
$Z$ on the left. Define a free left action of $H$ on $X\times Z$ by
\begin{equation} \label{eq:Haction}
h \cdot (x,z) \mapsto (x \cdot h^{-1}, h \cdot z)
\end{equation}
We denote by $X\times_{H} Z$ the quotient of $X\times Z$ by $H$.
 Let $q:X\times
Z\ra X\times_H Z$ the canonical quotient map. We define twisted
external product of sheaves as follows:

\begin{itemize} 
\item Let $\cF$ and $\cG$ be $H$-equivariant perverse sheaves on $X$
  and $Z$. Then $\cF\boxtimes \cG$ is a perverse sheaf on $X\times Z$
  equivariant with respect to the action \eqref{eq:Haction}.
  Thus, we obtain a canonical perverse sheaf on
  $X\times_H Z$. 

\item Suppose $\cG$ is as above, but $\cF$ is now a perverse sheaf on $Y$. Then $p^* \cF$ is an
  $H$-equivariant perverse sheaf on $X$. The construction of the
  previous paragraph applies to give us a perverse sheaf
  $\cF\tboxtimes \cG$ on $X\times_H Z$.  Roughly speaking,
  $\cF\tboxtimes \cG$ is $\cF$ ``along the base'' and $\cG$ ``along
  the fiber''.

\item By definition, $\tboxtimes$ is a functor 
  \begin{equation}\label{d:twistProd}
    \tboxtimes : \Perv_H(X) \times \Perv_H(Z) \ra \Perv(X\times_H Z),
    \quad \mathrm{satisfying} \,\, q^*\cF \cong (p^* \cF)\boxtimes \cG.
\end{equation} 

\item The functor $\tboxtimes$, with the property expressed in
  (\ref{d:twistProd}), makes sense in the following more general
  situation: $Y$ is a ``strict ind-scheme'' of ind-finite type over
  $k$, $Z$ is a strict ind-scheme of ind-finite type over $k$ equipped
  with a ``nice action'' of a pro-algebraic group $H$ over $k$. (For
  the notions ``strict ind-scheme'' and ``nice action'' see
  \cite{Gaitsgory01}.)  In this case, although $X$ need not be of
  ind-finite type, $X \times_H Z$ remains of ind-finite type, since it
  is a fibration over $Z$ with fibers isomorphic to $Y$.  The fact
  that it is nonetheless legitimate to work with $H$-equivariant
  perverse sheaves on $X$ is explained in \cite[\S
  2.2]{Nadler05}.
\item More generally, suppose that $H' < H$ is a proalgebraic subgroup
  such that $A := H/H'$ is an algebraic group.  Let $\cM_0$ be a
  multiplicative local system on $A$.  Let
  $\cM$ be the pullback of $\cM_0$ to $H$.  Suppose $\cF$ and $\cG$
  are $(H,\cM^{-1})$- and $(H,\cM)$-equivariant perverse sheaves on $X$ and $Z$,
  respectively; more precisely (to deal with the case that $X$ may not
  be ind-finite), we let $\cF$ be the pullback of an $(A,
  \cM_0^{-1})$-equivariant local system on $Y$ (cf.~Remark
  \ref{r:AEquiv}). Then $\cF\tboxtimes_{H'} \cG$ is (untwisted)
  equivariant with respect to the action of $A$ which descends from
  \eqref{eq:Haction}.  Hence, it descends to a canonical perverse
  sheaf $\cF\tboxtimes \cG$ on $X\times_H Z$. Thus, $\tboxtimes$
  also defines a functor
\[
\sP_{(H,\cM^{-1})}(X) \times \sP_{(H,\cM)} (Z) \ra \sP(X\times_H Z).
\]
\end{itemize}

\subsubsection{Twisted external product of local systems} 

Let $Y'$ (resp. $Z'$) denote a locally closed subscheme of $Y$
(resp. $Z$) of dimension $d$ (resp. $d'$). Let $X'\subseteq X$ denote
the restriction of the $G$-torsor $X$ to $Y'$. Let
$d'':=\dim(X'\times_{H} Z')$. Suppose $\cL$ is a local system on $Y'$
and $\cL'$ is an $H$-equivariant local system on $Z'$. The proof of
the following lemmas are left to the reader.

\begin{lemma} \label{l:prodLocalSys} $\cL[d]\tboxtimes \cL'[d']\cong
  \cL''[d'']$, where $\cL''$ is a local system on $X'\times_{H} Z'$.
 \end{lemma}

 \begin{lemma} \label{l:prodLocalSys2}
 Let
 \[
 j:X'\inj \overline{X'}, \quad \quad j': Z'\inj\overline{Z'},\quad \quad
 j'':X'\times_{H}Z' \inj \overline{X'\times_{H} Z'} .
 \] Assume that $j_{!}(\cL[d])$, $j_{!}'(\cL'[d'])$, and
 $j_{!}''(\cL''[d''])$ are perverse.\footnote{This assumption can probably be dropped if one defines twisted external products for equivariant complexes.} Then 
 \[
 j_{!}(\cL[d])\tboxtimes j_{!}'(\cL'[d']) \cong j_{!}''(\cL''[d'']).
 \] 
 \end{lemma}

\subsection{Trace of Frobenius} \label{ss:traceFrob} Let $\Fq$ be a
field with $q$ elements. Let $\bFq$ be an algebraic closure of
$\Fq$. The \emph{Frobenius substitution} $\varphi\in \Gal(\bFq/\Fq)$
is the automorphism $x\mapsto x^q$ of $\bFq$. The \emph{geometric
  Frobenius} $\Fr_q$, or simply the Frobenius, is the inverse of
$\varphi$.

Let $X$ be a separated scheme of finite type over $\Fq$. Let $x:\Spec
(\Fq)\ra X$ be an $\Fq$-point of $X$, and let $\bar{x}$ be a geometric
point lying above $x$. If $\cG\in \sD(X)$, then the fiber
$\cG_{\bar{x}}$ is a finite dimensional $\bQl$-vector space on which
$\Gal(\bFq/\Fq)$ acts \cite[\S 1.1.7]{WeilII}. We denote by
$\Tr(\Fr_q, \cG)(x)\in \bQl$ the trace of Frobenius acting on this
vector space. Thus, we obtain the trace function of $\cG$
\[
\Tr(\Fr_q, \cG):X(\Fq)\ra \bQl. 
\]
Similarly, we have trace functions $\Tr(\Fr_{q^n}, \cG): X(\Fqn) \to
\bQl$ for all $n \geq 1$; see \cite[\S 0.9 and \S
1.1.1]{Laumon87}. Note that with our conventions
\begin{equation}\label{eq:tate}
\Tr(\Fr_q, \cG(n))=q^{-n}\Tr(\Fr_q, \cG),
\end{equation}
where $\cG(n)$ denotes the $n^{\th}$ Tate twist of $\cG$ \cite[\S
E.1]{BDOrbit}.

\subsubsection{Character sheaves on connected commutative algebraic groups} \label{sss:multLocalSystem} Suppose $\cL$ is a
one-dimensional character sheaf on a connected algebraic group $G$ over
$\Fq$ (Definition \ref{d:multLocalSystem}). The property $m^*\cL \cong
\cL\boxtimes \cL$ ensures that the trace function is a one-dimensional
character $\Tr(\Fr_q, \cL):G(\Fq)\ra \bQlt$. For a general
noncommutative algebraic group, there may exist one-dimensional
characters of $G(\Fq)$ which do not arise in this manner; see, e.g.,
\cite[\S 1.5.5]{MityaThesis}.

If $G$ is \emph{commutative}, then every one-dimensional character of
$G(\Fq)$ can be obtained as the trace of Frobenius function of a
one-dimensional character sheaf on $G$.  To see this, let
\[
0\ra G(\Fq)\ra G \rar{\Fr_q-\id} G \ra 0
\]
denote the Lang central extension. Let $\eta: G(\Fq)\ra \bQlt$ be a 
character. Pushing forward the above central extension by $\eta^{-1}$,
we obtain a one-dimensional local system $\cN$ on $G$. One can check
that $\Tr(\Fr_q, \cN)=\eta$; see \cite[Sommes Trig]{Deligne77},
\cite[Example 1.1.3]{Laumon87}, and \cite[\S 1.8]{BDOrbit}.

\begin{remark} \label{r:norm} If $G$ is commutative, then for every
  integer $n$, we have a ``norm map'' $N_n: G(\Fqn)\ra
  G(\Fq)$. Namely, $N_n(x) = \prod_{i=0}^{n-1} \Fr_q^i(x)$, with the
  product taken in $G(\Fqn)$ (cf.~\cite[Sommes Trig, \S
  1.6]{Deligne77}).  Let $\eta_n:= \eta\circ N_n$. Then, one can show
  that $\Tr(\Fr_{q^n}, \cN\otimes_{\Fq} \Fqn) = \eta_n$; see \emph{op.~cit.} or \cite[\S
  1.1.3.3]{Laumon87}.
\end{remark}

\bibliographystyle{alpha}

\bibliography{../ref.geometrization}

\end{document}